\documentclass[11pt,leqno]{article}
\usepackage{graphicx}
\usepackage{amsfonts, amsthm}
\usepackage{amsxtra}
\usepackage[latin1]{inputenc}
\usepackage{fontenc}
\usepackage[dvips]{epsfig}
\usepackage[usenames]{color}
\begin{document}

\newtheorem{theo}{Theorem}[section]
\newtheorem{defi}[theo]{Definition}
\newtheorem{lemm}[theo]{Lemma}
\newtheorem{prop}[theo]{Proposition}
\newtheorem{rem}[theo]{Remark}
\newtheorem{exam}[theo]{Example}
\newtheorem{cor}[theo]{Corollary}
\newtheorem{state}[theo]{Statement}

\begin{center}
{\LARGE On the renormalizations of circle homeomorphisms with several break points} \\
\vspace{.25in}
\large{Abdumajid Begmatov\footnote{Institute of Mathematics, Academy of Science of the Republic of Uzbekistan, Do'rmon yo'li street 29, Akademgorodok, 100125 Tashkent, Uzbekistan. E-mail: abdumajidb@gmail.com}, Kleyber Cunha\footnote{Departamento de Matem\'{a}tica, Universidade Federal da Bahia, Av. Ademar de Barros s/n, CEP 40170-110, Salvador, BA, Brazil.  E-mail: kleyber@ufba.br} and
Akhtam Dzhalilov\footnote{Department of Natural and Mathematical Sciences, Turin Polytechnic University in Tashkent,
Niyazov Str. 17, 100095 Tashkent, Uzbekistan.  E-mail: a\_dzhalilov@yahoo.com}}

\end{center}

\begin{abstract}
Let $f$ be an orientation preserving homeomorphism on the circle with several break points, that is, its derivative $Df$ has jump discontinuities at these points. We study Rauzy-Veech renormalizations of piecewise smooth circle homeomorphisms by considering such maps as generalized interval exchange maps of genus one. Suppose that $Df$ is absolutely continuous on each interval of continuity and $D\ln{Df}\in \mathbb{L}_{p}$ for some $p>1$.  We prove that, under certain combinatorial assumptions on $f$, renormalizations $R^{n}(f)$ are approximated by piecewise M\"{o}bius functions in $C^{1+L_{1}}$-norm, that means, $R^{n}(f)$ are approximated in $C^{1}$-norm and $D^{2}R^{n}(f)$ are approximated in $L_{1}$-norm. In particular, if the product of the sizes of breaks of $f$ is trivial, then the renormalizations are approximated by piecewise affine interval exchange maps.
\end{abstract}

\sloppy

\section{Introduction}

One of the most studied classes of dynamical systems are orientation-preserving homeomorphisms of the circle $\mathbb{S}^{1}=\mathbb{R}/\mathbb{Z}$. Poincar\'{e} (1885) noticed that the orbit structure of an orientation-preserving diffeomorphism $f$ is determined by some irrational mod 1, the \emph{ rotation number} $\rho=\rho(f)$ of $f$, in the following sense: for any $x\in \mathbb{S}^1,$ the mapping $f^{j}(x)\rightarrow j\rho \mod 1,$ $j\in \mathbb{Z},$ is orientation-preserving. Denjoy proved, that if $f$ is an orientation-preserving $C^{1}$-diffeomorphism of the circle with  irrational rotation number $\rho$ and $\log f'$ has bounded variation then the orbit $\{f^{j}(x)\}_{j\in \mathbb{Z}}$  is dense and the mapping $f^{j}(x)\rightarrow j\rho \mod 1$  can therefore be extended by continuity to a homeomorphism $h$ of $\mathbb{S}^1,$ which conjugates $f$ to the linear rotation $f_{\rho}:x\rightarrow x+\rho \mod 1$. In this context it is a natural question to ask, under what conditions the conjugation is smooth. The first local results, that is the results requiring the closeness of diffeomorphism to the linear rotation, were obtained by Arnold \cite{Ar1961} and Moser \cite{Mo1966}. Next Herman \cite{He1979} obtained a first global result (i.e. not requiring the closeness of diffeomorphism to the linear rotation) asserting regularity of conjugation of the circle diffeomorphism. His result was developed by Yoccoz \cite{Yo1984}, Stark \cite{St1988}, Khanin \& Sinai \cite{KS1987, KS1989}, Katznelson \& Ornstein \cite{KO1989.1}, Khanin \& Teplinsky \cite{KT2009}. They have shown, that if $f$ is $C^3$ or $C^{2+\nu}$ and $\rho$ satisfies certain Diophantine condition, then the conjugation will be at least $C^1$. Notice that the renormalization approach used in \cite{KS1989} and \cite{St1988} is more natural in the spirit of Herman's theory. In this approach regularity of the conjugation can be obtained by using the convergence of renormalizations of sufficiently smooth circle diffeomorphisms. In fact, the renormalizations of a smooth circle diffeomorphism converge exponentially fast to a family of linear maps with slope 1. Such a convergence together with the condition on the rotation number (of Diophantine type) imply the regularity of conjugation.

The bottom of the scale of smoothness for a circle diffeomorphism $f$ was first considered by Herman in \cite{He1981}. He proved that if $Df$ is absolutely continuous, $D\log Df\in \mathbb{L}_{p}$ for some $p>1$, the rotation number $\rho=\rho(f)$ is irrational of bounded type (meaning that the entries in the continued fraction expansion of $\rho$ is bounded), and $f$ is close to the linear rotation $f_{\rho}$, then the conjugating map $h$ (between $f_{\rho}$ and $f$) is absolutely continuous. Later, using a martingale approach and not requiring the closeness of $f$ to the linear rotation, Katznelson and Ornstein \cite{KO1989.2} gave a different proof of Herman's theorem on absolute continuity of conjugacy. The latter condition on smoothness for $f$ (that is, $Df$ is absolutely continuous and $D\log Df\in \mathbb{L}_{p}$, $p>1$) will be called the Katznelson and Ornstein's (KO, for short) smoothness condition.

A natural generalization of diffeomorphisms of the circle are homeomorphisms with break points, i.e., those circle diffeomorphisms which are smooth everywhere with the exception of finitely many points at which their derivatives have jump discontinuities. Circle homeomorphisms with breaks were investigated by Herman \cite{He1979} in the piecewise-linear (PL) case. The studies of more general (non PL) circle diffeomorphisms with a unique break point started with the work of Khanin \& Vul \cite{KV1991}. It turns out that, the renormalizations of circle homeomorphisms with break points are rather different from those of smooth diffeomorphisms. Indeed, the renormalizations of such a  circle diffeomorphism converge exponentially fast to a two-parameter family of M\"{o}bius transformations. Applications of their result are very wide in many branches of one dimensional  dynamics, examples are the investigation of the invariant measures, nontrivial scalings and prevalence of periodic trajectories in one parameter families. In particular they investigated also the renormalization in the case of rational rotation number. Using convexity of the renormalization analysed positions of periodic trajectories of one parameter family of circle maps and they proved that the rotation number is rational for almost all parameter values. Moreover, the investigation of the  M\"{o}bius transformations in \cite{KhKhm2003},
\cite{KT2013} and \cite{KhYam} showed, that the renormalization operator in that space possesses hyperbolic properties analogous to those predicted by Lanford \cite{Lan1988} in the case of critical rotations. The result of Khanin and Vul is also at the core of the so-called \emph{rigidity problem}, which concerns the smoothness of conjugacy between two dynamical systems, which a priori are only topologically
equivalent. The rigidity problem for circle homeomorphisms with a break point has recently been completely solved in \cite{KhK2013}, \cite{KhK2014}, \cite{KhK2016}, \cite{KT2013}.

The next problem concerning the rigidity problem is to study the regularity properties of the conjugacy for circle maps with several break points. Circle maps with several break points can be considered as generalized interval exchange transformations of genus one. Marmi, Moussa and Yoccoz introduced in \cite{MMY2012} generalized interval exchange transformations, obtained by replacing the affine restrictions of generalized interval exchange transformations in each subinterval with smooth diffeomorphisms. They showed, that sufficiently smooth generalized interval exchange transformations of a certain combinatorial type, which are deformations of standard interval exchange transformations and tangent to them at the points of discontinuities, are smoothly linearizable.

Recently Cunha and Smania studied in \cite{CS2013} and \cite{CS2014} the Rauzy-Veech renormalizations of piecewise $C^{2+\nu}$- smooth circle homeomorphisms with several break points by considering such maps as generalized interval exchange transformations of genus one. They proved that Rauzy-Veech renormalizations of $C^{2+\nu}$- smooth generalized interval exchange maps satisfying a certain combinatorial condition are approximated by piecewise M\"{o}bius transformations in $C^2$- norm. Using convergence of renormalizations of two generalized interval exchange maps with the same bounded-type combinatorics and zero mean nonlinearities they proved in \cite{CS2014} that these maps $C^{1}$-smoothly conjugate to each other.

The purpose of the present work is to study the behavior of Rauzy-Veech renormalizations of generalized interval exchange maps of genus one and of low smoothness. We prove, that Rauzy-Veech renormalizations $R^{n}(f)$ of piecewise KO-smooth generalized interval exchange maps of genus one and satisfying certain combinatorial assumptions, are approximated by piecewise M\"{o}bius functions in $C^{1+L_{1}}$-norm, that means, the $R^{n}(f)$ are approximated in $C^{1}$-norm and the $D^{2}R^{n}(f)$ are approximated in $L_{1}$-norm. In particular, if  $f$ has zero mean nonlinearity, then the renormalizations are approximated by piecewise affine interval exchange maps.

Our main tool in this paper is an argument from real analysis which is used for $C^{2+\nu}$- smooth circle maps in \cite{KS1989}, \cite{KV1991} and for the KO-smooth case in \cite{BDM2014}. Note also that our proofs are based on considerations from the theory martingales, which for circle dynamics have been used by Katznelson and Onstein in \cite{KO1989.2}.

\section{Rauzy-Veech renormalization}

To describe the combinatorial assumptions of our results, we will introduce the Rauzy-Veech renormalization scheme. Let $I$ be an open bounded interval and $\mathcal{A}$ be an alphabet with $d\geq 2$ symbols. Consider the partition of $I$ into $d$ subintervals indexed by $\mathcal{A}$, that is, $\mathcal{P}=\{I_{\alpha},\,\, \alpha\in \mathcal{A}\}$. Let $f: I\rightarrow I$ be a bijection. We say that the triple $(f, \mathcal{A}, \mathcal{P})$ is a \textbf{generalized interval exchange map} with $d$ intervals (for short g.i.e.m.), if $f|_{I_\alpha}$ is an orientation-preserving homeomorphism for all $\alpha\in\mathcal{A}.$ Here and later, all intervals will be bounded, closed on the left and open on the right.

If $f|_{I_\alpha}$ is a translation, then $f$ is called a \textbf{standard interval exchange map} (for short s.i.e.m.). When $d=2$, identifying the endpoints of $I$, standard i.e.m.'s correspond to linear rotations of the circle and generalized i.e.m.'s to homeomorphisms of the circle with two break points.

Now we formulate some conditions on the combinatorics for g.i.e.m and define the renormalization scheme. Note that the combinatorial conditions and the renormalization scheme are the same for generalized and standard i.e.m. cases.

The order of the subintervals $I_{\alpha}$ before and after the map, constitutes the combinatorial data for $f$, which will be explicitly defined as follows.

Given two intervals $J$ and $U$, we will write $J <U$, if their interiors are disjoint and $x<y$, for every $x\in J$ and $y\in U$. This defines a partial order in the set of all intervals.

Let $f: I\rightarrow I$ be a g.i.e.m. with alphabet $\mathcal{A}$ and $\pi_0,\, \pi_1:\mathcal{A}\rightarrow \{1, . . . , d\}$, be bijections such that
$$
\pi_0(\alpha)<\pi_0(\beta),\,\,\,\,\,\, \mbox{\rm iff} \,\,\,\,\,\,\,\,I_{\alpha}<I_{\beta},
$$
and
$$
\pi_1(\alpha)<\pi_1(\beta),\,\,\,\,\,\, \mbox{\rm iff} \,\,\,\,\,\,\,\,f(I_{\alpha})<f(I_{\beta}).
$$

We call pair $\pi=(\pi_0, \pi_1)$ the \textbf{combinatorial data} associated to the g.i.e.m. $f$. We call $p=\pi_1^{-1}\circ\pi_0: \{1,...,d\}\rightarrow \{1,...,d\}$ the \textbf{monodromy invariant} of the pair $\pi=(\pi_0, \pi_1)$. When appropriate we will also use the notation $\pi = (\pi(1),\, \pi(2),\,...,\pi(d))$ for the combinatorial data of $f$. We always assume that the pair $\pi=(\pi_0, \pi_1)$ is \textbf{irreducible}, that is, for all $j\in \{1,...,d-1\}$ we have: $\pi_0^{-1}({1, . . . , j})\neq \pi_1^{-1}({1, . . . , j})$.

Let $\pi=(\pi_0, \pi_1)$ be the combinatorial data associated to the g.i.e.m $f$. For each $\varepsilon\in \{0, 1\}$, denote by $\alpha(\varepsilon)$ the last symbol in the expression of $\pi_{\varepsilon}$, that is \, $\alpha(\varepsilon)=\pi^{-1}_{\varepsilon}(d)$.

Let us assume that the intervals $I_{\alpha(0)}$ and $f(I_{\alpha(1)})$ have different lengths. Then the g.i.e.m. $f$ is called \textbf{Rauzy-Veech renormalizable}(renormalizable, for short). If $|I_{\alpha(0)}|>|f(I_{\alpha(1)})|$ we say that $f$ is renormalizable of \textbf{type} $\textbf{0}$. When $|I_{\alpha(0)}|<|f(I_{\alpha(1)})|$ we say that $f$ is renormalizable of \textbf{type} $\textbf{1}$.  In either case, the letter corresponding to the largest of these intervals is called \textbf{winner} and the one corresponding to the shortest is called the \textbf{loser} of $\pi$. Let $I^{(1)}$ be the subinterval of $I$ obtained by removing the loser, that is, the shortest of these two intervals:
$$
I^{(1)}=\left\{\begin{array}{ll}
I\setminus f(I_{\alpha(1)}),\,\,\, \mbox{\rm if} \,\,\,\,\, \mbox{\rm type 0}, \\
I\setminus I_{\alpha(0)},\,\,\, \mbox{\rm if} \,\,\,\,\, \mbox{\rm
type 1}.
\end{array}\right.
$$

Since the loser is the last subinterval on the right of $I$, the intervals $I$ and $I^{(1)}$ have the same left endpoint.

The \textbf{Rauzy-Veech induction} of $f$ is the first return map $R(f)$ to the subinterval $I^{(1)}$. We want to see $R(f)$ is again g.i.e.m. with the same alphabet $\mathcal{A}$. For this we need to associate to this map an $\mathcal{A}$ - indexed partition of its domain. Denote by $I^{(1)}_{\alpha}$ the subintervals of $I^{(1)}$. Let $f$ be renormalizable of type $0$. Then the domain of $R(f)$ is the interval $I^{(1)}=I\setminus f(I_{\alpha(1)})$ and
we have
\begin{equation}\label{Itype0}
I^{(1)}_{\alpha}=\left\{\begin{array}{ll}
I_{\alpha},\,\,\,\,\,\,\, \mbox{\rm for} \,\,\,\,\, \alpha\neq\alpha(0), \\
I_{\alpha(0)}\setminus f(I_{\alpha(1)}),\,\,\, \mbox{\rm for}
\,\,\,\,\, \alpha=\alpha(0).
\end{array}\right.
\end{equation}
These intervals form a partition of the interval $I^{(1)}$ and denoted by $\mathcal{P}^{(1)}=
\{I^{(1)}_{\alpha},\, \alpha\in \mathcal{A}\}$. Since $f(I_{\alpha(1)})$ is the last interval on the right of
$f(\mathcal{P})$, we have $f(I^{(1)}_{\alpha})\subset I^{(1)}$ for every $\alpha\neq \alpha(1)$. This means that, $R(f):=f$ restricted to these $I^{(1)}_{\alpha}$. On the other hand, due to $I^{(1)}_{\alpha(1)}=I_{\alpha(1)}$, we have
$$
f\left(I^{(1)}_{\alpha(1)}\right)=f\left(I_{\alpha(1)}\right)\subset I_{\alpha(0)},\,\,\,\, \mbox{\rm and so} \,\,\,\,
f^2\left(I^{(1)}_{\alpha(1)}\right)\subset f\left(I_{\alpha(0)}\right)\subset I^{(1)}.
$$
Then $R(f):=f^2$ restricted to $I^{(1)}_{\alpha(1)}$. Thus,
\begin{equation}\label{Rtype0}
R(f)(x)=\left\{\begin{array}{ll}
f(x),\,\,\,\,\, \mbox{\rm if} \,\,\,\,\, x\in I^{(1)}_{\alpha}\,\,\,\text{ and } \alpha\neq\alpha(1), \\
f^2(x),\,\,\,\, \mbox{\rm if} \,\,\,\,\, x\in I^{(1)}_{\alpha(1)}.
\end{array}\right.
\end{equation}

If $f$ is renormalizable of type $1$, the domain of $R(f)$ is the interval \ $I^{(1)}=I\setminus I_{\alpha(0)}$ and we have
\begin{equation}\label{Itype1}
I^{(1)}_{\alpha}=\left\{\begin{array}{lll}
I_{\alpha},\,\,\,\,\,\,\, \mbox{\rm for} \,\,\,\,\, \alpha\neq\alpha(0), \alpha(1), \\
f^{-1}(I_{\alpha(0)}),\,\,\, \mbox{\rm for} \,\,\,\,\,
\alpha=\alpha(0), \\
I_{\alpha(1)}\setminus f^{-1}(I_{\alpha(0)}),\,\,\, \mbox{\rm for}
\,\,\,\,\, \alpha=\alpha(1).
\end{array}\right.
\end{equation}
Then $f\left(I^{(1)}_{\alpha}\right)\subset I^{(1)}$ for every $\alpha\neq \alpha(0)$, and so $R(f)=f$ restricted to these $I^{(1)}_{\alpha}$. On the other hand,
$$
f^2\left(I^{(1)}_{\alpha(0)}\right)= f(I_{\alpha(0)})\subset I^{(1)},
$$
and, so $R(f)=f^2$ restricted to $I^{(1)}_{\alpha(0)}$. Thus,
\begin{equation}\label{Rtype1}
R(f)(x)=\left\{\begin{array}{ll}
f(x),\,\,\,\,\, \mbox{\rm if} \,\,\,\,\, x\in I^{(1)}_{\alpha}\,\,\,\text{ and } \alpha\neq\alpha(0), \\
f^2(x),\,\,\,\, \mbox{\rm if} \,\,\,\,\, x\in I^{(1)}_{\alpha(0)}.
\end{array}\right.
\end{equation}
It is easy to see, that $R(f)$ is a bijection on $I^{(1)}$ and an orientation-preserving homeomorphisms on each $I^{(1)}_{\alpha}$. Moreover, the alphabet $\mathcal{A}$  for $f$ and $R(f)$ remains the same.

The triple $(R(f), \mathcal{A}, \mathcal{P}^1)$ is called the \textbf{Rauzy-Veech renormalization} of $f$. If $f$ is
renormalizable of type $\varepsilon\in \{0, 1\}$, then the combinatorial data $\pi^1=(\pi_0^1, \pi_1^1)$ of $R(f)$ are given by
$$
\pi_{\varepsilon}^1:=\pi_{\varepsilon},\,\,\,\,\,\, \mbox{\rm
and}\,\,\,\,\,\,\,
\pi_{1-\varepsilon}^1(\alpha)=\left\{\begin{array}{lll}
\pi_{1-\varepsilon}(\alpha),\,\,\,\,\, \mbox{\rm if} \,\,\,\,\, \pi_{1-\varepsilon}(\alpha)\leq \pi_{1-\varepsilon}(\alpha(\varepsilon)), \\
\pi_{1-\varepsilon}(\alpha)+1,\,\,\,\,\, \mbox{\rm if} \,\,\,\,\, \pi_{1-\varepsilon}(\alpha(\varepsilon))< \pi_{1-\varepsilon}(\alpha)<d, \\
\pi_{1-\varepsilon}(\alpha(\varepsilon))+1,\,\,\,\,\, \mbox{\rm if}
\,\,\,\,\, \pi_{1-\varepsilon}(\alpha)=d.
\end{array}\right.
$$

We say that a g.i.e.m. $f$ is \textbf{infinitely renormalizable}, if $R^{n}(f)$ is well defined for every $n\in \mathbb{N}$. Let $I^{(n)}$ be the domain of $R^{n}(f)$. It is clear that, $R^{n}(f)$ is the first return map for $f$ to the interval $I^{(n)}$. Similarly, $R^{n}(f)^{-1}=R^{n}(f^{-1})$ is the first return map for $f$ to the
interval $I^{(n)}$.

For every interval of the form $J=[a, b)$ we put $\partial J:=\{a\}$.
\begin{defi}
We say that g.i.e.m. $f$  \textbf{has no connection}, if
\begin{equation}\label{Keane}
f^m(\partial I_{\alpha})\neq \partial I_{\beta},\,\,\,\,\,\,\,
\mbox{\it for all} \,\,\,\,\,\, m\geq 1\,\,\,\, \mbox{\it
and}\,\,\,\,\, \alpha,\, \beta\in\mathcal{A} \,\,\,\,\,\,\,\,\,
\mbox{\it with} \,\,\,\,\,\, \pi_{0}(\beta)\neq 1.
\end{equation}
\end{defi}

It is clear that in case $\pi_0(\beta)=1$ then $f(\partial I_{\alpha})=\partial I_{\beta}$ for $\alpha=\pi_1^{-1}(1)$. Notice  that the no connection condition is a necessary and sufficient condition for $f$ to be infinitely renormalizable. Condition (\ref{Keane}) means that the orbits of the left end point of the subintervals $I_{\alpha},\,\alpha\in\mathcal{A}$ are disjoint when ever they can be.

Let $\varepsilon_{n}$ be the type of the $n$-th renormalization and let $\alpha_n(\varepsilon_{n})$ the winner and
$\alpha_n(1-\varepsilon_{n})$ be the loser of the $n$-th renormalization.

\begin{defi}
We say that g.i.e.m. $f$ has $k$- \textbf{bounded combinatorics},
if for each $n\in \mathbb{N}$ and $\beta,\, \gamma\in\mathcal{A}$
there exist $n_1, p\geq0$ with $|n-n_1|<k$ and $|n-n_1-p|<k$ such
that
$$\alpha_{n_1}(\varepsilon_{n_1})=\beta,\,
\alpha_{n_1+p}(1-\varepsilon_{n_1+p})=\gamma,\,\,\, \mbox{\it and}
$$
$$
\alpha_{n_1+i}(1-\varepsilon_{n_1+p})=\alpha_{n_1+i+1}(\varepsilon_{n_1+i}),\,\,\,\,\,\,\,\,
\mbox{\it for every} \,\,\,\,\,\, 0\leq i<p.
$$
\end{defi}

We say that g.i.e.m. $f: I\rightarrow I$ has \textbf{genus one} (or belongs to the \textbf{rotation class}), if $f$ has at most two discontinuities. Note that every g.i.e.m. with either two or three intervals has genus one. The genus of g.i.e.m. is invariant under renormalization.

\begin{rem}
	Everey orientation-preserving homeomorphism of the circle when viewed as a g.i.e.m. with $d\geq2$ intervals, has genus one.
\end{rem}

\section{Main Results}

Denote by $\mathbb{B}^{KO}$  the set of g.i.e.m. satisfying the following conditions:
\begin{itemize}

\item[(i)] for each $\alpha\in \mathcal{A}$ we can extend $f$ to $\overline{I}_{\alpha}$ as an orientation-preserving diffeomorphism satisfying the Katznelson and Ornstein's (KO, for short) smoothness condition: $f'$ is absolutely continuous and $f''\in L_{p}$, for some $p>1$;

\item[(ii)] the map $f$ has no connection;

\item[(iii)] the map $f$ has $k$- bounded combinatorics and has genus one.
\end{itemize}

The main idea of the renormalization group method is to study the behaviour of the renormalization map $R^{n}(f)$ as $n\rightarrow\infty$. For this usually rescaling of the coordinates is used.

Let $H$ be a non-degenerate interval and  $g: H\rightarrow \mathbb{R}$  be a diffeomorphism. We define the \textbf{Zoom} (renormalized coordinate) $Z_{H}(g)$ of $g$ in $H$ as follows:
$$
Z_{H}(g)=\tau^{-1}\circ g \circ \tau,
$$
where $\tau: [0, 1]\rightarrow H$ is an orientation-preserving affine map.

Denote by $q^n_{\alpha}\in \mathbb{N}$ the first return time of the interval $I^{(n)}_{\alpha}$ to the interval $I^{(n)}$, that is, $R^{n}(f)|_{I^{n}_{\alpha}}=f^{q^n_{\alpha}}$, for some $q^n_{\alpha}\in \mathbb{N}$. Define the fractional linear map $F_{n}: [0, 1]\rightarrow [0, 1]$ as follows:
\begin{equation}\label{Fn}
F_{n}(x)=\frac{xm_{n}}{1+x(m_{n}-1)},\,\,\,\, \mbox{\rm where} \,\,\,\,\, m_{n}=\exp\left\{-\sum\limits_{i=0}^{q_{\alpha}^{n}-1}\int_{I^{(n)}_{\alpha}}\frac{f''(t)}{2f'(t)}dt\right\}.
\end{equation}

Whenever necessary, we will use $D^{m}f$ instead of the $m^{th}$ derivative of $f$. The first result of our present paper is the following

\begin{theo}\label{main1} Let $f\in \mathbb{B}^{KO}$. Then for all $\alpha\in
\mathcal{A}$ the following bounds hold:
$$
\|Z_{I^{(n)}_{\alpha}}(R^{n}(f))-F_{n}\|_{C^1[0,
1])}\leq \delta_{n},\,\,\,\,\,\,\,\,
\|Z_{I^{(n)}_{\alpha}}(D^2R^{n}(f))-D^2F_{n}\|_{L_1([0,
1], d\ell)}\leq \delta_{n},
$$
where $\delta_{n}=\mathcal{O}(\lambda^{n}+\eta_{n}),\,\, \lambda\in (0, 1)$ and $\eta_{n}\in l_{2}.$ 
\end{theo}

\smallskip

Denote by $\mathbb{B}^{KO}_{\star}$  the subset of functions $f\in \mathbb{B}^{KO}$ satisfying \textbf{zero mean nonlinearity} condition:
$$
\int\limits_{[0, 1]}\frac{f''(t)}{f'(t)}dt=0.
$$
Our second result is a consequence of Theorem \ref{main1}.

\begin{theo}\label{main2} Let $f\in \mathbb{B}^{KO}_{\star}$.
Then for all $\alpha\in
\mathcal{A}$ the following bounds hold:
$$
\|Z_{I^{(n)}_{\alpha}}(R^{n}(f))-Id\|_{C^1[0, 1])}\leq
\delta_{n},\,\,\,\,\,\,\,\,
\|Z_{I^{(n)}_{\alpha}}(D^2R^{n}(f))\|_{L_1([0, 1], d\ell)}\leq
\delta_{n},
$$
where $\delta_{n}=\mathcal{O}(\lambda^{\sqrt{n}}+\eta_{n}),\,\, \lambda\in (0, 1)$ and $\eta_{n}\in l_{2}.$ 
\end{theo}

\smallskip

\begin{rem} The sequence $\eta_{n}$ in Theorems \ref{main1} and \ref{main2} has an explicit form and is given in Proposition \ref{etan}.
\end{rem}

\begin{rem} The class $\mathbb{B}^{KO}$ is wider than $\mathbb{B}^{2+\nu}$ considered in \cite{CS2013}. However, the rate of approximations in Theorems \ref{main1} and \ref{main2} is not exponential, contrary to the class $\mathbb{B}^{2+\nu}$.
\end{rem}


The structure of the paper is as follows. In Section 4 we formulate some facts on dynamical partitions generated by interval exchange maps. Following Katznelson and Ornstein \cite{KO1989.2} we define a sequence of piecewise constant functions which generate a finite martingale. In Section 5 and Section 6, using the martingale expansion for the nonlinearity of $f$, we obtain some estimates for the sum of integrals of the nonlinearities of $f$.  Finally, in Section 7 we prove our main theorems.

\section{The dynamical partition and a martingale}

Let $(f, \mathcal{A}, \mathcal{P})$ be a g.i.e.m. with $d$ intervals and $\mathcal{P}=\{I_{\alpha}:
\alpha\in \mathcal{A}\}$ be the initial $\mathcal{A}$- indexed partition of $I$. For specificity we take $I=[0, 1)$. Suppose that $f$ is infinitely renormalizable. Let $I^{(n)}$ be the domain of $R^n(f)$. Note that $I^{(n)}$ is the nested sequence of subintervals, with the same left endpoint of $I$. We want to construct the dynamical partition of $I$ associated to the domain of $R^n(f)$.

As mentioned above, $R(f)$ is g.i.e.m. with $d$ intervals and the intervals $I^{(1)}_{\alpha}$ generate an $\mathcal{A}$- indexed partition of $I^{(1)}$, denoted by $\mathcal{P}^1$. By induction one can check, that $R^n(f)$ is g.i.e.m. with $d$ intervals. Let $\mathcal{P}^{n}=\{I^{(n)}_{\alpha}: \alpha\in \mathcal{A}\}$ be the $\mathcal{A}$- indexed partition of $I^{(n)}$, generated by $R^{n}(f)$. We call $\mathcal{P}^{n}$ the {\bf fundamental partition} and $I^{(n)}_{\alpha}$ the \textbf{fundamental segments} of rank $n$.

Since $R^n(f)$ is the first return map for $f$ to the interval $I^{(n)}$, each fundamental segment $I^{(n)}_{\alpha}\in \mathcal{P}^{n}$ returns to $I^{(n)}$ under certain iterates of the map $f$. Until returning, these intervals will be in the interval $I\setminus I^{(n)}$ for some time. Consequently the system of
intervals (their interiors are mutually disjoint)
$$
\xi_n=\left\{ f^{i}(I^{(n)}_{\alpha}),\,\, 0\leq i\leq q_{\alpha}^{n}-1,\,\, \alpha\in \mathcal{A}\right\}
$$
cover the whole interval and form a partition of $I$.

The system of intervals $\xi_n$ is called the \textbf{n-th dynamical partition} of $I$. The dynamical partitions $\xi_n$ are refined with increasing $n$, where $\xi_{n+1}\supset\xi_{n}$ means that
any element of the preceding partition is a union of a number of elements of the next partition, or belongs to the next partition. Denote by $\xi^{pr}_{n+1}$ the system of preserved intervals of $\xi_{n}$. More precisely, if $R^nf$ has type 0
$$\xi^{pr}_{n+1}=\left\{ f^{i}(I^{(n)}_{\alpha}),\,\, 0\leq i\leq q^{n}_{\alpha}-1,\,\, \mbox{~for~} \alpha\not=\alpha(0)\right\},$$
and if $R^nf$ has type 1

$$\xi^{pr}_{n+1}=\left\{ f^{i}(I^{(n)}_{\alpha}),\,\, 0\leq i\leq q^{n}_{\alpha}-1,\,\, \mbox{~for~} \alpha\not=\alpha(1)\right\}.$$

Let $\xi^{tn}_{n+1}:=\xi_{n+1}\setminus \xi^{pr}_{n+1}$ be the set of elements of $\xi_{n+1}$ which are properly contained in some element of $\xi_n.$ Therefore if $R^nf$ has type 0

\begin{eqnarray*}
\xi^{tn}_{n+1}&=&\left\{ f^i(I^{(n+1)}_{\alpha(0)}),\,\, 0\leq i< q^{n}_{\alpha(0)}\right\}\bigcup \left\{ f^i(I^{(n+1)}_{\alpha(1)}),\,\, 0\leq i< q^{n}_{\alpha(0)}\right\}\\
&=&\bigcup_{i=0}^{q^{n}_{\alpha(0)}-1}\left\{f^i\left(I^{(n)}_{\alpha(0)}\setminus f^{q^n_{\alpha(1)}}I^{(n)}_{\alpha(1)}\right)\right\}\bigcup \bigcup_{i=q^n_{\alpha(1)}}^{q^n_{\alpha(1)}+q^n_{\alpha(0)}-1}\left\{ f^i(I^{(n)}_{\alpha(1)})\right\},
\end{eqnarray*}

and if $R^nf$ has type 1

\begin{eqnarray*}
	\xi^{tn}_{n+1}&=&\left\{f^i(I^{(n+1)}_{\alpha(0)}),\,\, 0\leq i< q^{n}_{\alpha(1)}\right\}\bigcup\left\{f^i(I^{(n+1)}_{\alpha(1)}),\,\, 0\leq i< q^{n}_{\alpha(1)}\right\}\\
	& = & \bigcup_{i=0}^{q^n_{\alpha(1)}-1}\left\{f^i\left(f^{-q^n_{\alpha(1)}}(I^{(n)}_{\alpha(0)})\right)\right\}\bigcup \bigcup_{i=0}^{q^n_{\alpha(1)}-1}\left\{f^i\left(I^{(n)}_{\alpha(1)}\setminus f^{-q^n_{\alpha(1)}}(I^{(n)}_{\alpha(0)})\right)\right\}.
\end{eqnarray*}

So, the partition $\xi_{n+1}$ consists of the \emph{preserving} elements of $\xi_{n}$ and the \emph{images of two (new) intervals} for defining $R^{n+1}(f)$, that is, $\xi_{n+1}=\xi^{pr}_{n+1}\cup \xi^{tn}_{n+1}$. Note also that for the first return time $q^{n}_{\alpha}$, we have:
\begin{itemize}
\item[(1)] if \ $\alpha=\alpha^{n}(\varepsilon)$, then \ $q^{n+1}_{\alpha^{n}(\varepsilon)}=q^{n}_{\alpha^{n}(\varepsilon)}$;
\item[(2)] if \ $\alpha=\alpha^{n}(1-\varepsilon)$, then \ $q^{n+1}_{\alpha^{n}(1-\varepsilon)}=q^{n}_{\alpha^{n}(1-\varepsilon)}+q^{n}_{\alpha^{n}(\varepsilon)}$.
\end{itemize}

\bigskip

\textbf{Martingale.} Now we define a martingale generated by the dynamical partitions associated to g.i.e.m., and give its some properties which will be used in the proof of our results. A similar martingale generated by dynamical partitions associated to circle maps was considered in \cite{BDM2014} and \cite{KO1989.1}.

Let $g: I\rightarrow I$ be a function of class $\mathbb{L}_{p}(I, d\ell),\,\, p>1$. Using the dynamical
partitions $\xi_{n}$, we define a sequence of piecewise constant functions $\Phi_{n}: I\rightarrow R^{1},\,\,n\geq1$, on $I$ as follows
\begin{equation}\label{phin}
    \Phi_{n}(x):=\frac{1}{|\Delta^{(n)}|}
\int\limits_{\Delta^{(n)}}g(y)dy ,\,\,\ x\in
\Delta^{(n)},
\end{equation}
where $\Delta^{(n)}$ is an interval of the partition $\xi_{n}$.

\begin{theo}\label{martin1} Let $g\in \mathbb{L}_{p}(I, d\ell),\, p>1$. Then the sequence of piecewise functions
$\{\Phi_{n}(x),\, n\geq1\}$ generate a finite martingale with respect to the dynamical partition $\xi_{n}$.
\end{theo}
\begin{proof} Note that each $\Phi_{n}(x)$ is a step function, which takes constant values on each element $I^{(n)}_{\alpha}$ of the partition $\xi_{n}$. It follows that $\Phi_{n}(x)$ is $\xi_{n}$-
measurable.  Therefore, it is enough to show that
$$
E(\Phi_{n+1}/ \xi_n)=\Phi_{n},\,\, \mbox{\rm for all}\,\,\, n\geq 1,
$$
where $E(\Phi_{n+1}/\xi_n)$ is a conditional expectation of the random variable $\Phi_{n+1}$ with respect to the partition $\xi_n$. Define the characteristic functions on the elements of $\xi_{n}$:
$$
X^{(n)}_{\alpha, i}(x)=\left\{\begin{array}{ll} 1,\,\,\,
\mbox{\rm if}\,\,\,\, x\in f^{i}\left(I^{(n)}_{\alpha}\right),\\
0,\,\,\, \mbox{\rm if}\,\,\,\, x\notin f^{i}\left(I^{(n)}_{\alpha}\right).
\end{array}\right.
$$
where $\alpha\in\mathcal{A}$ and $0\leq i\leq q_{\alpha}^n-1$. By
definition of conditional expectation with respect to the partition, we have
\begin{equation}\label{condsum}
E(\Phi_{n+1}/ \xi_n)=\sum\limits_{\alpha\in
\mathcal{A}}\sum\limits_{i=0}^{q^{n}_{\alpha}-1}E\left(\Phi_{n+1}/f^{i}(I^{(n)}_{\alpha})\right)X^{(n)}_{\alpha,
i}(x).
\end{equation}
Recall, that the partition $\xi_{n+1}$ consists of the \emph{preserving} elements of $\xi_{n}$ and the \emph{images of two (new) intervals} for defining $R^{n+1}(f)$, that is, $\xi_{n+1}=\xi^{pr}_{n+1}\cup \xi^{tn}_{n+1}$. Split the sum (\ref{condsum}) in to two sums corresponding to $\xi^{pr}_{n+1}$ and $\xi^{tn}_{n+1}$:
\begin{equation}\label{twosum}
E(\Phi_{n+1}/
\xi_n)=\sum\limits_{J_{i}\in\xi^{pr}_{n+1}}E\left(\Phi_{n+1}/J_{i}\right)X^{(n)}_{\alpha,i}(x)+
\sum\limits_{J_{i}\in\xi^{tn}_{n+1}}E\left(\Phi_{n+1}/J_{i}\right)X^{(n)}_{\alpha,i}(x),
\end{equation}
where $J_{i}=f^{i}(I^{(n)}_{\alpha})$. Consider first the sum corresponding to $\xi^{pr}_{n+1}$ in (\ref{twosum}). Then
\begin{equation}\label{sum1}
E(\Phi_{n+1}/J_{i})=\int\limits_{[0,
1]}\Phi_{n+1}(x)\ell(dx/f^{i}(I^{(n)}_{\alpha}))=\frac{1}{|f^{i}(I^{(n)}_{\alpha})|}\int\limits_{f^{i}(I^{(n)}_{\alpha})}\Phi_{n+1}(x)dx=
\end{equation}
$$
=\frac{1}{|f^{i}(I^{(n)}_{\alpha})|}\int\limits_{f^{i}(I^{(n)}_{\alpha})}
\left(\frac{1}{|f^{i}(I^{(n)}_{\alpha})|}\int\limits_{f^{i}(I^{(n)}_{\alpha})}g(y)dy\right)dx=
\frac{1}{|f^{i}(I^{(n)}_{\alpha})|}\int\limits_{f^{i}(I^{(n)}_{\alpha})}g(y)dy.
$$
Next we consider the sum corresponding to $\xi^{tn}_{n+1}$ in (\ref{twosum}). Let $J_{i}:=\bigcup I^{(n+1)}_{\alpha}$, where $I^{(n+1)}_{\alpha}\in \xi^{tn}_{n+1}$. Then we obtain
$$
E(\Phi_{n+1}/J_{i})=\frac{1}{|f^{i}(I^{(n)}_{\alpha})|}
\int\limits_{f^{i}(I^{(n)}_{\alpha})}\Phi_{n+1}(x)dx=
\frac{1}{|f^{i}(I^{(n)}_{\alpha})|}\sum\limits_{I^{(n+1)}_{\alpha}\in J_{i}}\int_{I^{(n+1)}_{\alpha}}\Phi_{n+1}(x)dx=
$$
$$
=\frac{1}{|f^{i}(I^{(n)}_{\alpha})|}
\sum\limits_{I^{(n+1)}_{\alpha}\in J_{i}}\int_{I^{(n+1)}_{\alpha}}\left(\frac{1}{|I^{(n+1)}_{\alpha}|}
\int\limits_{I^{(n+1)}_{\alpha}}g(y)dy\right)dx=\frac{1}{|f^{i}(I^{(n)}_{\alpha})|}\int\limits_{f^{i}(I^{(n)}_{\alpha})}g(y)dy.
$$
This, and equations in (\ref{twosum}), (\ref{sum1}) imply the result.
\end{proof}

Denote by $\|f\|_{p}$ the norm of $f$ in $\mathbb{L}_{p}(I,\;d\ell),\,\ p>1.$

\begin{theo}\label{martin2} Let \, $g\in L_{p}(I,\, d\ell),\,\, p>1$. Then
$$
\lim\limits_{n\rightarrow \infty}\|g-\Phi_{n}\|_{2}=0.
$$
\end{theo}
\begin{proof} Note that the functions of the class $L_{p}$ are well approximated by continuous functions, that is, if \, $g\in L_{p}(I,\, d\ell),\,\, p>1$ \, then for any $\varepsilon>0$, there exist an uniformly continuous function $\omega_{\varepsilon}$ and a summable function $\psi_{\varepsilon}$ such that
$$
g(x)=\omega_{\varepsilon}(x)+\psi_{\varepsilon}(x),\,\,\, x\in I,\,\,\,\,\, \mbox{\rm and moreover},\,\,\,\  \|\psi_{\varepsilon}\|_{2}<\varepsilon.
$$

Consider the partition $\xi_{n}$. Then $[0, 1]=\bigvee\limits_{\alpha\in
\mathcal{A}}{\bigvee\limits_{i=0}^{q^n_{\alpha}-1}f^i(I^{(n)}_{\alpha})}$.
Using the above expansion for $g$ we get
$$
\|g-\Phi_{n}\|^{2}_{L_{2}}=\int\limits_{[0, 1]}|g(x)-\Phi_{n}(x)|^2dx=
\sum\limits_{\alpha\in
\mathcal{A}}\sum\limits_{i=0}^{q_{\alpha}^{n}-1}\int\limits_{f^i(I^{(n)}_{\alpha})}|\Phi_{n}(x)-
(\omega_{\varepsilon}(x)+\psi_{\varepsilon}(x))|^{2}dx\leq
$$
$$
\leq 2\sum\limits_{\alpha\in
\mathcal{A}}\sum\limits_{i=0}^{q_{\alpha}^{n}-1}\int\limits_{f^i(I^{(n)}_{\alpha})}
\left|\frac{1}{|f^i(I^{(n)}_{\alpha})|}\int\limits_{f^i(I^{(n)}_{\alpha})}\omega_{\varepsilon}(y)dy-
\omega_{\varepsilon}(x)\right|^{2}dx+
$$
$$
+2\sum\limits_{\alpha\in
\mathcal{A}}\sum\limits_{i=0}^{q_{\alpha}^{n}-1}\int\limits_{f^i(I^{(n)}_{\alpha})}
\left|\frac{1}{|f^i(I^{(n)}_{\alpha})|}\int\limits_{f^i(I^{(n)}_{\alpha})}\psi_{\varepsilon}(y)dy-
\psi_{\varepsilon}(x)\right|^{2}dx:=M^{(1)}_{n}+M^{(2)}_{n}.
$$
It is clear that
$$
M^{(2)}_n\leq 2\sum\limits_{\alpha\in
\mathcal{A}}\sum\limits_{i=0}^{q_{\alpha}^{n}-1}
\int\limits_{f^i(I^{(n)}_{\alpha})}\left|\frac{1}{|f^i(I^{(n)}_{\alpha})|}
\int\limits_{f^i(I^{(n)}_{\alpha})}\psi_{\varepsilon}(y)dy|^{2}dx+
\int\limits_{f^i(I^{(n)}_{\alpha})}|\psi_{\varepsilon}(x)\right|^{2}dx\leq
4\|\psi_{\varepsilon}\|^{2}_{2}.
$$
By assumption, $\omega_{\varepsilon}$ is uniformly continuous. This means, that for all \, $x, y: |x-y|<\delta$ inequality $|\omega_{\varepsilon}(x)-\omega_{\varepsilon}(y)|<\varepsilon$ is fulfilled. On the other had, for each $f^{i}(I^{(n)}_{\alpha})\in \xi_{n}$, we have \,  $\max\limits_{\alpha,\, i}|f^{i}(I^{(n)}_{\alpha})|\leq \lambda^{n},\,\, \lambda\in (0, 1)$ (see  for instance (\ref{normxinm})). It follows that for all $x,\, y\in f^{i}(I^{(n)}_{\alpha})$, the inequality $|\omega_{\varepsilon}(x)-\omega_{\varepsilon}(y)|<\varepsilon$ is
fulfilled. Then
$$
M^{(1)}_n\leq\sum\limits_{\alpha\in
\mathcal{A}}\sum\limits_{i=0}^{q_{\alpha}^{n}-1}\int\limits_{f^i(I^{(n)}_{\alpha})}
\left|\frac{1}{|f^i(I^{(n)}_{\alpha})|}\displaystyle\int\limits_{f^i(I^{(n)}_{\alpha})}\left(\omega_{\varepsilon}(y)-
\omega_{\varepsilon}(x)\right)dy\right|^{2}dx\leq \varepsilon^{2}.
$$
The estimates for $M^{(1)}_{n}$ and $M^{(2)}_{n}$ imply the assertion of Theorem \ref{martin2}.
\end{proof}

Set \, $\Phi_{0}(x)=\Phi_{0} =\int\limits_{[0, 1]}g(y)dy,\,\, \mbox{\rm for all}\,\, x\in I$. Define \,
$h_{n}:=\Phi_{n}-\Phi_{n-1},\,\, n\geq1$.

\begin{theo}\label{martin3} Let \, $g\in L_{p}(I,\, d\ell),\,\, p>1$. Then

(1) \,\ $g-\Phi_{0}=\sum\limits_{n=1}^{\infty}h_{n}\;\;
(\mbox{\rm in}\; \mathbb{L}_{2}- \mbox{\rm norm});$\\

(2)\,\,\ for any interval $\Delta^{(n-1)}$  of the partition
$\xi_{n-1}$ and for all $n\geq1$, we have
$$\int\limits_{\Delta^{(n-1)}}h_{n}(x)d\ell=0.$$
\end{theo}
\begin{proof} Assertion $(1)$ immediately follows from Theorem \ref{martin2}. We'll prove the second assertion. Consider the partition $\xi_{n-1}$. Recall, that $\xi_{n}=\xi^{pr}_{n}\cup \xi^{tn}_{n}$. Let $\Delta^{(n-1)}\in \xi_{n-1}$. If $\Delta^{(n-1)}\in \xi^{pr}_{n}$, then we have
$$
\int\limits_{\Delta^{(n-1)}}h_{n}(x)d\ell=\int\limits_{\Delta^{(n-1)}}\Phi_{n}(x)dx-
\int\limits_{\Delta^{(n-1)}}\Phi_{n-1}(x)dx=0.
$$
Suppose that $\Delta^{(n-1)}\notin \xi^{pr}_{n}$. Let $I^{(n)}_{\alpha}\in \Delta^{(n-1)}$ and $I^{(n)}_{\alpha}\in \xi^{tn}_{n}$. Then we obtain
$$
\int\limits_{\Delta^{(n-1)}}h_{n}(x)d\ell=\int\limits_{\Delta^{(n-1)}}\Phi_{n}(x)dx-
\int\limits_{\Delta^{(n-1)}}\Phi_{n-1}(x)dx=
$$
$$
=\sum\limits_{I^{(n)}_{\alpha}\in \Delta^{(n-1)}}\int\limits_{I^{(n)}_{\alpha}}\left(\frac{1}{|I^{(n)}_{\alpha}|}
\int\limits_{I^{(n)}_{\alpha}}g(y)dy\right)dx-
\int\limits_{\Delta^{(n-1)}}\left(\frac{1}{|\Delta^{(n-1)}|}\int\limits_{\Delta^{(n-1)}}g(y)dy\right)dx=
$$
$$
=\sum\limits_{I^{(n)}_{\alpha}\in \Delta^{(n-1)}}\int\limits_{I^{(n)}_{\alpha}}g(y)dy-
\int\limits_{\Delta^{(n-1)}}g(y)dy=0.
$$
We are done.
\end{proof}

The following theorem plays an important role for our result.

\begin{theo}\label{martin4}(see. \cite{KO1989.2}) Suppose \, $g\in L_{p}(I,\, d\ell),\,\, 1< p\leq 2$. Let $\{(\Phi_{n}(x),  n\geq1\}$ be a $\mathbb{L}_{p}$- bounded martingale w.r.t. the partition $\xi_{n}.$ Then the sequence $\{||h_{n}\|_{p}\,\,, n\geq1\}$ belongs to $l_{2}$.
\end{theo}

We need the following lemma which can be checked easily.

\begin{lemm}\label{deltan} Let \, $\{r_{n},\,\ n\geq 1\}\in l_{2}$ be a  sequence of positive numbers
and let \, $\lambda\in (0, 1)$ be a constant. Set \, $\epsilon_n:=\sum\limits_{j=n}^{\infty}\lambda^{j-n}r_j,\,\,\ n\geq1.$ Then $\sum\limits_{n=1}^{\infty}\epsilon^2_{n}<\infty$.
\end{lemm}

As we know, in case of \, KO smoothness, the function $\frac{f''}{f'}$ is defined almost everywhere. Whenever necessary, let's \textit{conditionally} call the derivative $\frac{f''}{f'}$ the \textbf{nonlinearity} of $f$.

Next we define the a sequence of piecewise constant functions for $g=\frac{f''}{f'}$ in a similar way as in (\ref{phin}):
$$
\Phi_{n}(x):=\frac{1}{|\Delta^{(n)}|}
\int\limits_{\Delta^{(n)}}\frac{f''(t)}{f'(t)}dt ,\,\,\ x\in
\Delta^{(n)},\,\,\,\, \alpha\in\mathcal{A}
$$
where $\Delta^{(n)}$ is an interval of the partition $\xi_{n}$. Set $h_n=\Phi_n-\Phi_{n-1}$. Similar to results in Theorem \ref{martin1}, Theorem \ref{martin4} and Lemma \ref{deltan} we obtain the following

\begin{prop}\label{etan} Let $f\in \mathbb{B}^{KO}$ and $\eta_n=\sum\limits_{m=n}^{\infty}\lambda^{m-n}\|h_m\|_p.$
Then $\{\eta_n\}\in l_2$.
\end{prop}

\textbf{Bounded geometry or Denjoy type inequalities.} Denote by $\mathbb{B}^{1+bv}$ the set of g.i.e.m $f: I\rightarrow I$ satisfying the conditions $(ii)-(iii)$, which are piecewise $C^1$- smooth and have bounded variation of the first derivative.

From now on we will denote by $C$ constants, which depend only on the original map $f$. Put $x_{i}=f^{i}(x),\,\,\, i\geq0$ \, and $x_0:=x$. The following lemma plays a key role in studying metrical properties of the dynamical partition $\xi_{n}$.

\begin{lemm}\label{DRnf}(see \cite{CS2013}) Let $f\in \mathbb{B}^{1+bv}$. Put \, $\theta:=\mathrm{Var}_{I}\log f'$. Then there is a constant $C>0$ such that
$$
e^{-C\theta}\leq \prod\limits_{i=0}^{q_{\alpha}^{n}-1}Df(x_{i})\leq e^{C\theta},\,\,\, \mbox{\rm for all} \,\,\,\, x\in I^{(n)}_{\alpha}.
$$
\end{lemm}

Define the norm of the dynamical partition $\xi_n$ by
$$
\|\xi_n\|=\max\{|f^i(i^n_{\alpha})|\},\,\, \mbox{\it where the maximum is
taken for all}\,\,\, \alpha\in \mathcal{A}\,\, \mbox{\it and}\,\,\,\,
0\leq i\leq q_{\alpha}^{n}-1.
$$
Using lemma \ref{DRnf},  it has been shown in \cite{CS2013} that the intervals of the dynamical partition $\xi_n$ have exponentially small length.

\begin{lemm}\label{normxink}(see \cite{CS2013}) Let $f\in \mathbb{B}^{1+bv}$. Then for sufficiently large $n$ there is $\lambda\in (0, 1)$ such that \ $\|\xi_{n+k}\|\leq \lambda\|\xi_n\|.$
\end{lemm}

The following corollary follows from Lemma \ref{normxink}.

\begin{cor}\label{normxinm} Let $f\in \mathbb{B}^{1+bv}$. Then for sufficiently large $n$ and $m$ with
\, $m-n>k$, there is $\lambda\in (0, 1)$ such that
\begin{equation}\label{normxinm1}
\|\xi_{n}\|\leq \lambda^{\frac{n}{k}-1}\,\,\, \mbox{\it and}
\,\,\,\, \|\xi_{m}\|\leq \lambda^{\frac{m-n}{k}-1}\|\xi_{n}\|.
\end{equation}
\end{cor}

Consider the sequence of dynamical partitions $\xi_n$. We recall the following definition introduced in \cite{KO1989.1}.

\begin{defi}\label{qnsmall} An interval $J=(a, b)\subset[0, 1]$ is called {\bf $q_n$-small} and its end points $a, b$ are {\bf $q_n$-close}, if the system of intervals $f^i(J),\,\, 0\leq i\leq q_n$ are disjoint.
\end{defi}

The following lemmas are modification of similar ones used in \cite{KO1989.1} and \cite{KO1989.2} for circle maps.

\begin{lemm}\label{ellsum} Suppose that $f\in \mathbb{B}^{1+bv}$. Let $I^{(n)}_{\alpha}$ be $q_n$- small and
$m<n-k$, then
$$
\ell\left(\bigcup\limits_{i=0}^{q_{m+1}-1}f^i(I^{(n)}_{\alpha})\right)\leq
C\lambda_{1}^{n-m},\,\,\, \mbox{\rm where} \,\,\, \lambda_{1}= \lambda^{1/k}.
$$
\end{lemm}
\begin{proof} Let $I^{(m+1)}_{\beta}$ be $q_{m+1}$- small and assume that it contains the interval $I^{(n)}_{\alpha}$. The second inequality in (\ref{normxinm1}) implies:  $$|f^{i}(I^{(n)}_{\alpha})|\leq
C\lambda_{1}^{n-m}|f^{i}(I^{(m)}_{\beta})|,\,\,\, 0\leq i\leq q_{m+1}.$$
Then, we get
$$
\ell\left(\bigcup\limits_{i=0}^{q_{m+1}-1}f^i(I^{(n)}_{\alpha})\right)=
\sum\limits_{i=0}^{q_{m+1}-1}|f^i(I^{(n)}_{\alpha}))|\leq
C\lambda_{1}^{n-m}\sum\limits_{i=0}^{q_{m+1}-1}|f^i(I^{(m)}_{\beta}))|\leq
C\lambda_{1}^{n-m}.
$$
\end{proof}

Put $Df^i(t):=(f^i)'(t)$.

\begin{lemm}\label{finzi} Suppose that $f\in \mathbb{B}^{1+bv}$. Let $x$  and $y$ be $q_n$- close. Then for
any $0\leq l\leq q_n-1$ the following inequality holds:
$$
e^{-\theta} \leq \frac{Df^{l}(x)}{Df^{l}(y)}\leq e^{\theta}.
$$
\end{lemm}
\begin{proof} Take any two $q_{n}$-close points $x, y\in [0, 1]$ and $0\le m\le q_{n}-1$. Denote by $I^{(n)}_{\alpha}$ the open interval with endpoints $x$ and $y$. Since the intervals $f^{i}(I^{(n)}_{\alpha}),\ 0\leq i<q_{n}$ are disjoint, we obtain
$$
|\log Df^{m}(x)-\log Df^{m}(y)|\le \sum_{s=0}^{q_{n}-1}| \log
Df(f^{s}(x))- \log Df(f^{s}(y))|\le \theta.
$$
From this, we obtain the result.
\end{proof}

Consider an arbitrary fundamental segment $I^{(n)}_{\alpha}$ of the $n$-th basic partition $\mathcal{P}^{(n)}$.
Put $I^{(n)}_{\alpha}=[a, b]$. For each $0\leq i\leq q_{\alpha}^{n}$, we introduce the relative
coordinates  $z_i: [f^i(a), f^i(b)]\rightarrow [0, 1]$ as:
\begin{equation}\label{zi}
z_{i}:=\frac{f^{i}(x)-f^{i}(a)}{f^{i}(b)-f^{i}(a)},\,\,\,\,\, x\in
[a, b].
\end{equation}

We consider the relative coordinates $z_i$ as functions of the variable $z_0$.

\begin{lemm}\label{propzi} Suppose that $f\in \mathbb{B}^{KO}$. Then for all \, $i=0,1,...,(q_{\alpha}^n-1)$
the following inequalities hold:
\begin{equation}\label{propzi1}
e^{-2\theta}\leq \frac{z_0(1-z_0)}{z_i(1-z_i)}\leq
e^{2\theta},\,\,\, e^{-\theta}\leq \frac{dz_i}{dz_0}\leq
e^{\theta},\,\,\, \int\limits_{0}^{1}\left|\frac{d^2z_i}{dz_0^2}\right|dz_0\leq
C \left\|\frac{f''}{f'}\right\|_1.
\end{equation}
\end{lemm}
\begin{proof}
Using (\ref{zi}) we get
$$
\frac{z_0(1-z_0)}{z_i(1-z_i)}=\frac{x-a}{f^i(x)-f^i(a)}\cdot
\frac{b-x}{f^i(b)-f^i(x)}\cdot
\left(\frac{f^i(b)-f^i(a)}{b-a}\right)^2=
\frac{Df^i(t_0)}{Df^i(t_1)}\cdot \frac{Df^i(t_0)}{Df^i(t_2)},
$$
where \, $t_0\in[a, b],\,\, t_1\in [a, x],\,\, t_2\in [x, b]$. Note that both of the pairs $\{t_0,\, t_1\}$ and $\{t_0,\, t_2\}$ are $q_n$-close. Applying Lemma \ref{finzi}, we obtain the first inequality in
(\ref{propzi1}).

Using (\ref{zi}) we find for $\dfrac{dz_i}{dz_0}$:
$$
\frac{dz_i}{dz_0}=\frac{dz_i}{dx_i}\cdot \frac{dx_i}{dx}\cdot
\frac{dx}{dz_0}=\frac{|I^{(n)}_{\alpha}|}{|I^{(n)}_{\alpha,
i}|}\cdot Df^i(x)=\frac{Df^i(x)}{Df^i(t_0)},\,\,\, \mbox{\rm
where}\,\,\,\, t_0\in I^{(n)}_{\alpha}.
$$
Then due to Lemma \ref{finzi}, we get the second inequality in (\ref{propzi1}).

Note, that the functions $\frac{d^2z_i}{dz_0^2}$ are defined almost everywhere. We can estimate the functions $\frac{d^2z_i}{dz_0^2}$ in the integral norm. According to the relations $x=a+z_0(b-a)$ and $x_i=f^i(x)$, we find for $\frac{d^2z_i}{dz_0^2}$:
$$
\frac{d^2z_i}{dz_0^2}=\frac{d}{dx_i}\left(\frac{|I^{(n)}_{\alpha}|}{|I^{(n)}_{\alpha,\,
i}|}\cdot
\prod\limits_{j=0}^{i-1}f'(x_j)\right)\cdot\frac{dx_i}{dx}\cdot\frac{dx}{dz_0}=
\frac{|I^{(n)}_{\alpha}|^2}{|I^{(n)}_{\alpha,\, i}|}\cdot
Df^i(x)\cdot \sum\limits_{j=0}^{i-1}\frac{f''(x_j)}{f'(x_j)}\cdot
Df^j(x)=
$$
$$
=\frac{dz_i}{dz_0}\cdot\left(\sum\limits_{j=0}^{i-1}\frac{f''(x_j)}{f'(x_j)}\cdot
Df^j(x)\right)\cdot |I^{(n)}_{\alpha}|=\frac{dz_i}{dz_0}\cdot\left(\sum\limits_{j=0}^{i-1}\frac{f''(x_j)}{f'(x_j)}\cdot
|I^{(n)}_{\alpha,\, j}|\cdot \frac{dz_j}{dz_0}\cdot\right).
$$
This together with the first and second inequalities in (\ref{propzi1}) imply
$$
\int\limits_{0}^{1}\left|\frac{d^2z_i}{dz_0^2}\right|dz_0\leq
e^{3\theta}\int\limits_{0}^{1}\left(\sum\limits_{j=0}^{i-1}\left|\frac{f''(x_j)}{f'(x_j)}\right|\cdot
|I^{(n)}_{\alpha, j}|\right)dz_i.
$$
Substituting  $z_i=\dfrac{x_i-a_i}{b_i-a_i}$ in the last integral, we obtain
$$
\int\limits_{0}^{1}\left|\frac{d^2z_i}{dz_0^2}\right|dz_0\leq
C\sum\limits_{j=0}^{q_{\alpha}^{n}-1}\int\limits_{a_j}^{b_j}|f''(x_j)|dx_j
\leq C \left\|\frac{f''}{f'}\right\|_1,
$$
as we claimed.
\end{proof}

\section{Approximations of the nonlinearity for $\mathbb{B}^{KO}$ maps with a martingale}

In the low smoothness case considered here, we still have not known, how to obtain the necessary bounds for the integral of $\frac{f''}{f'}$ on any interval of the dynamical partition. For this reason we had to consider the sum of these integrals over all the intervals of dynamical partition.

Let $I^{(n)}_{\alpha}$ be an arbitrary fundamental segment of the $n$-th basic partition $\mathcal{P}^{(n)}$. Let $I^{(n)}_{\alpha}=[a, b]$. For the iteration of the interval $I^{(n)}_{\alpha}$ and its endpoints we
use the following notations:
$$
I^{(n)}_{\alpha, i}=f^{i}(I^{(n)}_{\alpha})=[a_i,
b_i],\,\, 0\leq i\leq q_{\alpha}^{n}-1,
$$
where $a_0=a,\,\, b_0=b$ and
$$
a_{i}=f^{i}(a),\,\, b_{i}=f^{i}(b),\,\,\, x_{i}=f^{i}(x)\in[a_i,
b_i].
$$
For simplicity of the notation put $q_{n}:=q_{\alpha}^{n}$. Next define
$$
S^{(1)}_{n}:=\sum\limits_{i=0}^{q_{n}-1}\int\limits_{a_i}^{x_i}\frac{f^{\prime
\prime}(t)}{f^{\prime}(t)}\left(\frac{t-a_i}{x_i-a_i}-\frac{1}{2}\right)dt.
$$

\begin{prop}\label{modSn1}
Let $f\in \mathbb{B}^{KO}$. Then we have $|S^{(1)}_{n}|=\mathcal{O}(\lambda^{n}+\eta_n)$, where $\lambda\in(0, 1)$ and $\eta_{n}\in l_{2}$ is from Proposition \ref{etan}.
\end{prop}
\begin{proof}
In order to use Theorem \ref{martin3} for $g=\frac{f''}{f'}$, we rewrite the sum $S^{(1)}_{n}$ as follows
\begin{equation}\label{sn1}
S^{(1)}_{n}=\sum\limits_{i=0}^{q_{n}-1}\int\limits_{a_i}^{x_i}\left(\frac{f^{\prime
\prime}(t)}{f^{\prime}(t)}-\Phi_0-\sum\limits_{m=1}^{N}h_m(t)\right)\left(\frac{t-a_i}{x_i-a_i}-\frac{1}{2}\right)dt+
\end{equation}
$$
+\sum\limits_{i=0}^{q_{n}-1}\int\limits_{a_i}^{x_i}\sum\limits_{m=1}^{N}h_m(t)\left(\frac{t-a_i}{x_i-a_i}-\frac{1}{2}\right)dt.
$$
It is easy to see that the absolute value of the first sum in (\ref{sn1}) is not greater than
$$
C\|\frac{f^{\prime
\prime}}{f^{\prime}}-\Phi_0-\sum\limits_{m=1}^{N}h_{m}\|_{1}.
$$
The first assertion of Theorem \ref{martin3} implies, that
$$
\lim\limits_{N\rightarrow\infty}\|\frac{f^{\prime \prime}}{f^{\prime}}-\Phi_0-\sum\limits_{m=1}^{N}h_{m}\|_{1}=0.
$$ Then we can choose a sufficiently large number $N$
such that
$$\left\|\frac{f^{\prime
\prime}}{f^{\prime}}-\Phi_0-\sum\limits_{m=1}^{N}h_{m}\right\|_{1}\leq
\lambda^{n}.
$$
Hence, the absolute value of the first sum in (\ref{sn1}) is bounded above by $C\lambda^{{n}}$.

Recall, that the point $x_i=f^i(x)$ belongs to the interval $[a_i, b_i].$ Next choose $r_{0}>0$ minimal, such that for $[a_i, b_i]$ one has
\begin{equation}\label{r0}
I^{(n+r_0+1)}_{\beta}\subset[a_i, x_i]\subset I^{(n+r_0)}_{\alpha}\subset[a_i, b_i],
\end{equation}
where $I^{(n+r_0+s)}\in \xi_{n+s},$ for $s=0,1$.

To estimate the last sum in (\ref{sn1}), we split the sum in the integrand into three terms corresponding to the summations over $1\le m\le n+r_{0},\,\, n+r_{0}+1 \le m \le n+r_{0}+k$ and  $n+r_0+k+1\le m\le N$. Consider the first  sum. By definition, the function $h_{m}(t)$ takes constant values on the atoms of the dynamical partition $\xi_{m}.$ On the other hand, when passing from partition $\xi_{m}$ to $\xi_{m+1}$, the elements of the partition $\xi_{m}$ are preserved, or divided in two subintervals. This together with $[a_i, x_i]\subset I^{(n+r_0)}\subset I^{(n)}_{\alpha,\, i}\in \xi_n$ imply that the function $h_{m}(t)$ takes constant values on the intervals $[a_i, b_i],$ i.e. $ h_{m}([a_i,\; b_i])=h_{m}(a_i),\,\, i=0, 1, ..., q^{\alpha}_n$. Using these remarks, we get
$$
\sum\limits_{i=0}^{q_{n}-1}
\int\limits_{a_i}^{x_i}\sum\limits_{m=1}^{n+r_0}h_{m}(t)\left(\frac{t-a_i}{x_i-a_i}-\frac{1}{2}\right)dt=
\sum\limits_{m=1}^{n+r_0}\sum\limits_{i=0}^{q_{n}-1}h_{m}(a_i)
\int\limits_{a_i}^{x_i}\left(\frac{t-a_i}{x_i-a_i}-\frac{1}{2}\right)dt=0.
$$
Consider the sum over $n+r_0+1 \leq m\leq n+r_{0}+k$. Then we have
$$
\sum\limits_{m=n+r_{0}+1}^{n+r_0+k}\sum\limits_{i=0}^{q_{n}-1}
\int\limits_{a_i}^{x_i}h_{m}(t)\left(\frac{t-a_i}{x_i-a_i}-\frac{1}{2}\right)dt\le
\sum\limits_{m=n+r_{0}+1}^{n+r_0+k}\sum\limits_{i=0}^{q_{n}-1}
\int\limits_{a_i}^{b_i}|h_{m}(t)|dt\le \sum\limits_{m=n+r_{0}+1}^{n+r_0+k}\|h_{m}\|_{p}.
$$
It is easy to see, that the last sum also belongs to the class $l_{2}$.

Next we consider the sum over $n+r_0+k+1 \leq m\leq N$ and denote this sum by $P_{n}$. Since $m\geq n+r_{0}+k+1,$ each atom $[a_i, b_i]\in \xi_{n}$ is the union of intervals of the partition $\xi_{m-1}.$ Define a piecewise constant function $L^{m, i}(y)$ on $[a_i, b_i]$ which takes constant values on the atoms of the partition $\xi_{m-1}$, such that
$$ L^{m, i}|_{[c^{(k-1)}, \;
d^{(k-1)}]}= \frac{c^{(k-1)}-a_i}{x_i-a_i}-\frac{1}{2},
$$
if $[c^{(k-1)},\; d^{(k-1)}]\in \xi_{m-1}$  and $[c^{(k-1)},\; d^{(k-1)}]\subset[a_i,\;  x_i]$. Then we rewrite the
sum $P_{n}$ as follows
$$
P_n=\sum\limits_{m=n+r_0+k+1}^{N}\sum\limits_{i=0}^{q_{n}-1}
\int\limits_{a_i}^{x_i}h_{m}(t)\left[\frac{t-a_i}{x_i-a_i}-\frac{1}{2}-L^{m,i}(t)\right]dt+
$$
$$
+\sum\limits_{m=n+r_0+k+1}^{N}\sum\limits_{i=0}^{q_{n}-1}
\int\limits_{a_i}^{x_i}h_{m}(t)L^{m,i}(t)dt.
$$
Denote by $P_{n}^{(1)}$ and $P_{n}^{(2)}$ the last two sums over $m$, respectively. First we estimate the sum $P_{n}^{(2)}.$ Since $[a_i, x_i]\subset [a_i, b_i]\in \xi_{n},$ the interval $[a_i, x_i]$ is covered by intervals of the partition $\xi_{m-1}.$ Denote by $\bar I_{i}^{(m-1)}$ the interval of the partition $\xi_{m-1}$ containing the point $x_{i}$. If there are two such intervals then we consider the left one. Applying the second assertion of Theorem \ref{martin3},
we obtain:
$$ |P_{n}^{(2)}|\le
\left|\sum\limits_{m=n+r_0+k+1}^{N}\sum\limits_{i=0}^{q_{n}-1}
     \sum\limits_{I^{(m-1)}\subset [a_i, b_i]}
     L^{m, i}(I^{(m-1)})\int\limits_{I^{(m-1)}}h_{m}(t)dt\right|+
$$
$$ +\sum\limits_{m=n+r_0+k+1}^{N}\sum\limits_{i=0}^{q_{n}-1}
L^{m,i}(\bar I_i^{(m-1)})\int\limits_{\bar
I_{i}^{(m-1)}}|h_{m}(t)|dt\le C \sum\limits_{m=n+r_0+k+1}^{\infty}
\int\limits_{U_{m}}|h_{m}(t)|dt,
$$
where $U_m=\bigcup\limits_{i=0}^{q^{\alpha}_{n}-1}f^{i}(\bar I_{0}^{(m-1)})$. Lemma \ref{ellsum} implies that \ $\ell (U_m)\leq
\lambda_{1}^{m-n-1}$. We have
$$
\sum\limits_{m=n+r_0+k+1}^{\infty}\int\limits_{U_{m}}|h_{m}(t)|dt\le
\sum\limits_{m=n+k+1}^{\infty}\|h_{m}\|_p(\ell(U_m))^{\frac{1}{q}}\leq
C\sum\limits_{m=n+1}^{\infty}\lambda^{m-n-1}_{2}\|h_{m}\|_p=\eta_n,
$$
where $\lambda_2=\lambda_{1}^{1/kq}$. Finally, $|P_{n}^{(2)}|\le \eta_n$ and $\{\eta_n\}\in l_2$, due to Proposition \ref{etan}.

Since $m\geq n+r_0+k+1$, Corollary \ref{normxinm} implies that
\begin{equation}\label{inimpor}
\left|\frac{t-a_i}{x_i-a_i}-\frac{1}{2}-L_{m,
i}(t)\right|=\left|\frac{t-c^{(m-1)}}{x_i-a_i}\right|\le \left|\frac{d^{(m-1)}-c^{(m-1)}}{I^{(n+r_{0}+1)}_{\beta}}\right| \le
\lambda^{m-n-r_0-2}_1,
\end{equation}
for all \, $t\in[c^{(m-1)}, d^{(m-1)}]\subset[a_i, x_i]$, with $\lambda_{1}=\lambda^{1/k}$. Using this estimate, we obtain:
$$|P_{n}^{(1)}|\le
\sum\limits_{m=n+r_0+k+1}^{N}\lambda^{m-n-r_0-2}_1\sum\limits_{i=0}^{q_{n}-1}\int\limits_{I^{(n+r_0)}(i)}|h_{m}(t)|dt\le
$$
$$
\le\sum\limits_{m=n+r_0+k+1}^{N}\lambda^{m-n-r_0-2}_1\cdot\lambda_{1}^{\frac{r_0}{q}}
\cdot\left(\int\limits_{U_{r_0}}|h_{m}(t)|^{p}dt\right)^{1/p}
\le C \sum\limits_{m=n}^{\infty}\lambda^{m-n}_2\|h_{m}\|_p= \eta_n,
$$
where we have used
$$
U_{r_0}=\bigcup\limits_{i=0}^{q_{n}-1}I^{(n+r_0)}(i),\,\,\, \mbox{\rm and} \,\,\, \ell(U_{r_0})<\lambda_{1}^{n+r_0-n}=\lambda_{1}^{r_0}.
$$
By Proposition \ref{etan}, $\{\eta_n\}\in l_2$. This completes the proof.
\end{proof}

Set
$$
E_n:= \sum\limits_{i=0}^{q_n-1}\left((1-z_i)\int\limits_{a_i}^{x_i}\frac{f''(t)}{f'(t)}\frac{t-a_i}{x_i-a_i}dt-
z_i\int\limits_{x_i}^{b_i}\frac{f''(t)}{f'(t)}\frac{b_i-t}{b_i-x_i}dt\right).
$$

\begin{rem}\label{modEn} Using the same arguments for estimating $|S^{(1)}_n|$, one can show that $|E_n|=\mathcal{O}(\lambda^{n}+\eta_n)$.
\end{rem}

Now we define
$$
Q_n:=\sum\limits_{i=0}^{q_n-1}\int\limits_{a_i}^{b_i}\left|\int\limits_{a_i}^{x_i}
\frac{f''(t)}{f'(t)}\frac{t-a_i}{(x_i-a_i)^2}dt-
\int\limits_{x_i}^{b_i}\frac{f''(t)}{f'(t)}\frac{b_i-t}{(b_i-x_i)^2}dt\right|dx_i.
$$

\begin{prop}\label{modQn} Let $f\in \mathbb{B}^{KO}$. Then we have $|Q_{n}|=\mathcal{O}(\lambda^{n}+\eta_n)$, where $\lambda\in(0, 1)$ and $\eta_{n}\in l_{2}$ is from Proposition \ref{etan}.
\end{prop}
\begin{proof} It is clear that
$$
Q_n\leq \sum\limits_{i=0}^{q_n-1}\int\limits_{a_i}^{b_i}
\left|\int\limits_{a_i}^{x_i}\left(\frac{f''(t)}{f'(t)}-\Phi_0-
\sum\limits_{m=1}^{N}h_m(t)\right)\frac{t-a_i}{(x_i-a_i)^2}dt-
\right.$$
\begin{equation}\label{in}
\left.-\int\limits_{x_i}^{b_i}\left(\frac{f''(t)}{f'(t)}-\Phi_0-
\sum\limits_{m=1}^{N}h_m(t)\right)\frac{b_i-t}{(b_i-x_i)^2}dt\right|dx_i+
\end{equation}
$$
+\sum\limits_{i=0}^{q_n-1}\int\limits_{a_i}^{b_i}\left|\frac{1}{x_i-a_i}
\int\limits_{a_i}^{x_i}\sum\limits_{m=1}^{N}h_m(t)\frac{t-a_i}{x_i-a_i}dt-
\frac{1}{b_i-x_i}\int\limits_{x_i}^{b_i}\sum\limits_{m=1}^{N}h_m(t)\frac{b_i-t}{b_i-x_i}dt
\ \right| dx_i.
$$
Denote by $Q^{(1)}_n$ and $Q^{(2)}_n$ the last two sums over $i$ in (\ref{in}), respectively. Let us first estimate $Q^{(1)}_{n}.$ Using H\"{o}lder's inequality for the integrals over $[a_i, x_i]$ and $[x_i, b_i]$ in $Q^{(1)}_n$
we get
$$
Q^{(1)}_n\leq
\frac{4}{\sqrt{3}}\sum\limits_{i=0}^{q_n-1}\left(\int\limits_{a_i}^{b_i}
\left|\frac{f''(t)}{f'(t)}-\Phi_0-\sum\limits_{m=1}^{N}h_m(t)\right|^2dy\right)^{1/2}\cdot\sqrt{b_i-a_i}.
$$
Again using H\"{o}lder's inequality for the last sum we obtain:
$$
Q^{(1)}_n\leq
\frac{4}{\sqrt{3}}\left\|\frac{f''}{f'}-\Phi_0-\sum\limits_{m=1}^{N}h_m\right\|_2.
$$
The first assertion of Theorem \ref{martin3} implies that $$\lim\limits_{N\rightarrow+\infty}\left\|\frac{f''}{f'}-\Phi_0-\sum\limits_{m=1}^{N}h_m\right\|_2=0.
$$
Then we choose sufficiently a large number $N$ such that
$$\left\|\frac{f''}{f'}-\Phi_0-\sum\limits_{m=1}^{N}h_m\right\|_2\leq\lambda^{n}_2.
$$
Hence $Q^{(1)}_n$ is bounded above by $C\lambda^{n}_2$.

To estimate $Q^{(2)}_{n},$  we split the  integrand into three terms with summations over $1\le m\le n+r_0,\,\, n+r_{0}+1 \le m \le n+r_{0}+k$, and $n+r_0+k+1\le m\le N$, where $r_0$ was defined in (\ref{r0}). Denote the corresponding sums by $T_{1},\, T_{2},\, T_{3}$. Then $Q^{(2)}_{n}\le T_{1}+T_{2}+T_{3}$. Consider first the sums over  $m$  from $1$ to $n+r_{0}.$ The piecewise constant function $h_m(t)$ takes constant values on the  atoms of the partition $\xi_{m}$. Since $[a_i, x_i]\subset I^{(n+r_0)}\subset I^{(n)}_{\alpha, \, i}\in \xi_n$, the function $h_m(t)$ takes constant values on the intervals  $[a_i, b_i],\,\,\, 1\leq m \leq n+r_{0}$, i.e.
$h_{m}([a_i,\; b_i])=h_{m}(a_i)$. Then we have
$$
T_{1}=\sum\limits_{i=0}^{q_n-1}\int\limits_{a_i}^{b_i}\left|
\int\limits_{a_i}^{x_i}\sum\limits_{m=1}^{n+r_0}h_{m}(y)\frac{y-a_i}{(x_i-a_i)^2}dy-
\int\limits_{x_i}^{b_i}\sum\limits_{m=1}^{n+r_0}h_{m}(y)\frac{b_i-y}{(b_i-x_i)^2}dy\right|dx_i=
$$
$$
=\sum\limits_{i=0}^{q_n-1}\sum\limits_{m=1}^{n+r_0}h_{m}(a_i)\int\limits_{a_i}^{b_i}
\left|\int\limits_{a_i}^{x_i}\frac{y-a_i}{(x_i-a_i)^2}dy-
\int\limits_{x_i}^{b_i}\frac{b_i-y}{(b_i-x_i)^2}dy\right|dx_i=0,
$$
where we used, that the difference of two integrals in the last sum vanishes.

Consider next the sum $T_{2}$. Using Holder's inequality for the integral and for the sum we obtain:
$$
T_{2}\le \sum\limits_{m=n+r_{0}+1}^{n+r_{0}+k}\sum\limits_{i=0}^{q_n-1}\int\limits_{a_i}^{b_i}\left[\left(
\int\limits_{a_i}^{x_i}|h_m(t)|^{p}dt\right)^{\frac{1}{p}}(x_{i}-a_{i})^{\frac{1}{q}-1}\right]dx_{i}+
$$
$$
+\sum\limits_{m=n+r_{0}+1}^{n+r_{0}+k}\sum\limits_{i=0}^{q_n-1}\int\limits_{a_i}^{b_i}\left[
\left(\int\limits_{x_i}^{b_i}|h_m(t)|^{p}dt\right)^{\frac{1}{p}}(b_{i}-x_{i})^{\frac{1}{q}-1}\right]dx_{i}\le
$$
$$
\le 2q\sum\limits_{m=n+r_{0}+1}^{n+r_{0}+k}\sum\limits_{i=0}^{q_n-1}\left(
\int\limits_{a_i}^{b_i}|h_m(t)|^{p}dt\right)^{\frac{1}{p}}(b_{i}-a_{i})^{\frac{1}{q}}\le 2q\sum\limits_{m=n+r_{0}+1}^{n+r_{0}+k}\|h_{m}\|_{p}.
$$
Since $\{\|h_{m}\|_{p}\}\in l_{2}$ and $k$ is a fixed number, the last sum also belongs to $l_{2}$.

Next we consider the sum $T_{3}$, i.e. the sum over  $n+r_0+k+1 \leq m \leq N$.  Notice, that each interval $[a_i, b_i ]\in \xi_n$ is the union of a finite number of intervals of the partition $\xi_{m-1}$. Define piecewise constant functions $\widetilde{L}^{m, i}(t)$  and $\widetilde{M}^{m, i}(t)$ on $[a_i, x_i]$ and $[x_i, b_i]$, respectively, which are approximations of the integrands $Q^{(2)}_{n}$ in the corresponding intervals and take constant values on the atoms of $\xi_{m-1},$ as follows
$$
\widetilde{L}^{m, i}|_{[c^{(m-1)},
d^{(m-1)}]}=\frac{c^{(m-1)}-a_i}{x_i-a_i},
$$
if  $[c^{(m-1)}, d^{(m-1)}]\in \xi_{m-1}$ and  $[c^{(m-1)}, d^{(m-1)}]\subset [a_{i}, x_{i}],$ respectively
$$
\widetilde{M}^{m, i}|_{[c^{(m-1)}, d^{(m-1)}]}=\frac{b_i-d^{(m-1)}}{b_i-x_i},
$$
if  $[c^{(m-1)}, d^{(m-1)}]\in \xi_{m-1}$ and $[c^{(m-1)}, d^{(m-1)}]\subset [x_{i}, b_i]$.

Then we have
$$
T_{3}\leq \sum\limits_{m=n+r_0+k+1}^{N}\sum\limits_{i=0}^{q_n-1}\int\limits_{a_i}^{b_i}
\left|\frac{1}{x_i-a_i}\int\limits_{a_i}^{x_i}h_m(t)\left(\frac{t-a_i}{x_i-a_i}-\widetilde{L}^{m,
i}(t)\right)dt\right|dx_i+
$$
\begin{equation}\label{jn}
+\sum\limits_{m=n+r_0+k+1}^{N}\sum\limits_{i=0}^{q_n-1}\int\limits_{a_i}^{b_i}\left|\frac{1}{b_i-x_i}
\int\limits_{x_i}^{b_i}h_m(t)\left(\frac{b_i-t}{b_i-x_i}-\widetilde{M}^{m,
i}(t)\right)dt\right|dx_i+
\end{equation}
$$
+\sum\limits_{i=0}^{q_n-1}\int\limits_{a_i}^{b_i}\left|\int\limits_{a_i}^{x_i}\sum\limits_{m=n+r_0+k+1}^{N}h_m(t)\cdot\frac{\widetilde{L}^{m, i}(t)}{x_i-a_i}dt-
\int\limits_{x_i}^{b_i}\sum\limits_{m=n+r_0+k+1}^{N}h_m(t)\cdot\frac{\widetilde{M}^{m,
i}(t)dt}{b_i-x_i}\right|dx_i
$$
Denote by $J^{(1)}_{n}$,  $J^{(2)}_{n}$,  $J^{(3)}_{n}$ the three double sums in (\ref{jn}). Consider first the sum $J^{(3)}_{n}$. Recall, that $[a_i, x_i]\subset [a_i, b_i]$ and the interval $[a_i, x_i]$ is covered by intervals of the partition $\xi_{m-1}$. If $x_i$ lies on the boundary of one of the intervals of $\xi_{m-1}$, then by the second assertion of Theorem \ref{martin3}  $J^{(3)}_n$ vanishes. If the point $x_i$ lies inside of some interval of $\xi_{m-1}$, we denote this interval by $\widetilde{I}^{(m-1)}(x_i)$ and get
$$
J_n^{(3)}\leq
\sum\limits_{m=n+r_0+k+1}^{N}\sum\limits_{i=0}^{q_{n}-1}\int\limits_{a_i}^{b_i}\left|\frac{1}{x_i-a_i}\sum\limits_{I^{(m-1)}
\subset[a_i, x_i]}^{''}\widetilde{L}^{m,i}(I^{(m-1)})
\int\limits_{I^{(m-1)}}h_m(t)dt\right|dx_i+
$$
$$+\sum\limits_{m=n+r_0+k+1}^{N}\sum\limits_{i=0}^{q_{n}-1}\int\limits_{a_i}^{b_i}\frac{1}{x_i-a_i}\widetilde{L}^{m,i}
(\widetilde{I}^{(m-1)}(x_i))\int\limits_{\widetilde{I}^{(m-1)}(x_i)}
|h_m(t)|dtdx_i+
$$
$$
+\sum\limits_{m=n+r_0+k+1}^{N}\sum\limits_{i=0}^{q_{n}-1}\int\limits_{x_i}^{b_i}\left|\frac{1}{b_i-x_i}\sum\limits_{I^{(m-1)}
\subset[x_i, b_i]}^{''}\widetilde{M}^{m,i}(I^{(m-1)})
\int\limits_{I^{(m-1)}}h_m(t)dt\right|dx_i+
$$
$$
+\sum\limits_{m=n+r_0+k+1}^{N}\sum\limits_{i=0}^{q_{n}-1}\int\limits_{a_i}^{b_i}\frac{\widetilde{M}^{m,i}
(\widetilde{I}^{(m-1)}(x_i))}{b_i-x_i}\int\limits_{\widetilde{I}^{(m-1)}(x_i)}
|h_m(t)|dtdx_i=(I)+(II)+(III)+(IV).
$$
By the second assertion of Theorem \ref{martin3} the sums (I) and (III) are equal to zero. We estimate only sum (II), the sum (IV) is estimated analogously. Note that the step function $\widetilde{L}^{m,i}$ is bounded above by $1$. Using H\"{o}lder's inequality for the second (interior) integral in (II), we obtain:
$$
\int\limits_{\widetilde{I}^{(m-1)}(x_i)} |h_m(t)|dt\leq
\left(\int_{\widetilde{I}^{(m-1)}(x_i)}
|h_m(t)|^{p}dt\right)^{1/p}\cdot |\widetilde{I}^{(m-1)}(x_i)|^{1/q}.
$$
We take the maximum of the integral
$$
\int\limits_{\widetilde{I}^{(m-1)}(x_i)} |h_m(t)|^{p}dt
$$
over $x_i$. Then, after multiplying with $|\widetilde{I}^{(m-1)}(x_i)|^{1/q}$, we simplify as follows
$$
\frac{|\widetilde{I}^{(m-1)}(x_i)|^{1/q}}{x_i-a_i}\leq
\frac{|\widetilde{I}^{(m-1)}(x_i)|^{1/q}}{|I^{n+r_0}|^{1/q}}\cdot
(x_i-a_i)^{\frac{1}{q}-1}\leq \lambda^{m-n-r_0-2}_2\cdot
(x_i-a_i)^{\frac{1}{q}-1},
$$
where we used Corollary \ref{normxinm} and the definition of $r_0$. After these preparations we have
$$
(II)\leq q\sum\limits_{m=n+r_0+k+1}^{N}\lambda^{m-n-r_0-2}_2\sum\limits_{i=0}^{q_{n}-1}
\max\limits_{a_i\leq x_i\leq
b_i}\left(\int_{\widetilde{I}^{(m-1)}(x_i)}
|h_m(t)|^{p}dt\right)^{1/p}(b_i-a_i)^{\frac{1}{q}}\leq
$$
$$
\leq
C\sum\limits_{m=n+r_0+k+1}^{N}\lambda^{m-n-r_0-2}_2\|h_m\|_p\leq C\sum\limits_{m=n}^{\infty}\lambda^{m-n}_2\|h_m\|_p=\eta_n.
$$
Finally, due to the Proposition \ref{etan},  $|J_n^{(3)}|\leq \eta_n$ and $\{\eta_n\}\in l_2$.

We next estimate $J^{(1)}_n$ in (\ref{jn}), $J^{(2)}_n$ is estimated analogous. Using inequality (\ref{inimpor}) and H\"{o}lder's inequality for the interior integral over $[a_i, x_i]$  in $J^{(1)}_n$, we obtain
$$
\int\limits_{a_i}^{b_i}
\left|\frac{1}{x_i-a_i}\int\limits_{a_i}^{x_i}h_m(t)\left(\frac{t-a_i}{x_i-a_i}-\widetilde{L}^{m,
i}(t)\right)dt\right|dx_i\leq
$$
$$
\leq\lambda_{1}^{m-n-r_0-3}\left(\int\limits_{a_i}^{b_i}|h_m(t)|^{p}dt\right)^{1/p}
\int\limits_{a_i}^{b_i}(x_i-a_i)^{\frac{1}{q}-1}dx_i=
$$
$$
=q\lambda_{1}^{m-n-r_0-3}\left(\int\limits_{a_i}^{b_i}|h_m(t)|^{p}dt\right)^{1/p}
(b_i-a_i)^{\frac{1}{q}}.
$$
Then, using H\"{o}lder's inequality for the sum over $i$ in $J^{(1)}_n$, we get
$$
|J_n^{(1)}|\leq C\sum\limits_{m=n+r_0+k+1}^{\infty}\lambda_{1}^{m-n-r_0-1}
\left(\sum\limits_{i=0}^{q_{n}-1}\int\limits_{a_i}^{b_i}|h_m(t)|^{p}dt\right)^{1/p}
\left(\sum\limits_{i=0}^{q_{n}-1}(b_i-a_i)\right)^{1/q}\leq\eta_n.
$$
We are done.
\end{proof}

Set
$$
U_{n}:=\sum\limits_{i=0}^{q_n-1}\int\limits_{a_i}^{b_i}\left[\int\limits_{a_i}^{x_i}\left(f''(x_i)-f''(t)\right)
\frac{t-a_i}{(x_i-a_i)^2}dt+\int\limits_{x_i}^{b_i}\left(f''(x_i)-f''(t)\right)
\frac{b_i-t}{(b_i-x_i)^2}dt\right]dx_{i}.
$$

\begin{rem}\label{modUn} Using the same arguments for estimating $Q_{n}$, one can show, that  $|U_n|=\mathcal{O}(\lambda^{n}+\eta_n)$. Note, that here the differences of $f''$ in $U_{n}$ allow us to use the martingale expansion.
\end{rem}

\section{Estimates for $\tau_n(z_0)$}

In this section we will obtain some estimates for the sum $\tau_n(z_0)$ defined in (\ref{psitau}). More specifically, the estimates for $\tau_n(z_0)$ are reduced to the estimates in Propositions \ref{modSn1}, \ref{modQn} and Remarks \ref{modEn}, \ref{modUn}.

Define
\begin{equation}\label{Ai}
A_i:={-}\frac{\frac{1}{f'(a_i)(x_i-a_i)}\cdot\int\limits_{a_{i}}^{x_{i}}f''(t)(t-a_i)dt+
\frac{1}{f'(a_i)(b_i-x_i)}\cdot\int\limits_{x_i}^{b_i}f''(t)(b_i-t)dt}{1+
\frac{1}{f'(a_i)(b_i-a_i)}\int\limits_{a_i}^{b_i}f''(t)(b_i-t)dt},
\end{equation}
\begin{equation}\label{psitau}
N_{i}:=\int\limits_{a_i}^{b_i}\frac{f^{\prime \prime} (t)}{2f^{\prime}
(t)}dt,\,\,\,\,\,\, \psi_i(z_{0})=N_{i}-\log\Big(\frac{1+A_{i}z_i}{1+A_i(z_{i}-1)}\Big),
\,\,\,\,\, \tau_{n}(z_0):=\sum_{i=0}^{q_{n}-1}\psi_i(z_0).
\end{equation}

\begin{prop}\label{estimtaun} Suppose that $f\in \mathbb{B}^{KO}$. Then the following estimates hold for $\tau_n(z_0)$ and its derivatives
\begin{equation}\label{modtaun} \max\limits_{0\leq
z_0\leq1}{|\tau_n(z_0)|}\leq \delta_n,\,\,\,\,\,\, \max\limits_{0\leq
z_0\leq1}{|(z_0-z^{2}_0)\tau'_n(z_0)|}\leq \delta_n,
\end{equation}
\begin{equation}\label{modtaun1}
\int\limits_{[0, 1]}|\tau'_n(z_0)|dz_0\leq \delta_n,\,\,\,\,\,\,
\int\limits_{[0, 1]}|(z_0-z^{2}_0)|\tau''_n(z_0)|dz_0\leq \delta_n,
\end{equation}
where $\delta_{n}=\mathcal{O}(\lambda^{n}+\eta_n),\, \lambda\in(0, 1)$ and $\eta_{n}\in l_{2}$ is from proposition \ref{etan}.
\end{prop}
\begin{proof} Denote by $V_i$ the second term in the denominator of $A_i$. Using H\"{o}lder's inequality we get
$$
|V_i|\leq\frac{1}{f'(a_i)(b_i-a_i)}\int\limits_{a_i}^{b_i}|f''(t)|(b_i-t)dt\leq
C(b_i-a_i)^{1/q},\,\,\,\,\, \mbox{\rm where} \,\,\,\,\,
q=\frac{p}{p-1}.
$$
In analogy one can show, that the absolute values of both terms of the numerator of $A_{i}$ are bounded by $C(b_i-a_i)^{1/q}$. Since $[a_i, b_i]$ is an interval of the partition $\xi_{n}$, by corollary \ref{normxinm} its length is not larger than $C\lambda^{\frac{n}{k}}$. Hence
$$
|V_i|=\mathcal{O}(\lambda^{n}_1),\,\,\, |A_i|=\mathcal{O}(\lambda^{n}_1),\,\,\, \mbox{\rm where} \,\,\, \lambda_1=\lambda^{\frac{1}{k}}.
$$
We rewrite $\tau_n(z_0)$ as follows
\begin{equation}\label{taun}
\tau_n(z_0)=\sum\limits_{i=0}^{q_{n}-1}N_{i}-\sum\limits_{i=0}^{q_{n}-
1} \log\left(\frac{1+A_{i}z_{i}}{1+A_{i}(z_{i}-1)}\right)=-\log m_n-
\sum\limits_{i=0}^{q_{n}-1}A_i-\sum\limits_{i=0}^{q_{n}-1}O(A_{i}^{2}).
\end{equation}
One can then estimate the last sum in (\ref{taun}) as follows
$$
\left|\sum\limits_{i=0}^{q_{n}-1}A_{i}^{2}\right|\leq
2\sum\limits_{i=0}^{q_{n}-1}
\frac{1}{(1+V_i)^2}\left\{\left(\int\limits_{a_i}^{x_i}\frac{f''(t)(t-a_i)}{f'(a_i)(x_i-a_i)}dt\right)^2+
\left(\int\limits_{x_i}^{b_i}\frac{f''(t)(b_i-t)}{f'(a_i)(b_i-x_i)}dt\right)^2\right\}\leq
$$
$$
\leq C\|f''\|_{p}\sum\limits_{i=0}^{q_{n}-1}\left\{(x_i-a_i)\int\limits_{a_i}^{x_i}|f''(t)|dt+
(b_i-x_i)\int\limits_{x_i}^{b_i}|f''(t)|dt\right\}\leq
$$
\begin{equation}\label{sumAi2}
\leq C\cdot \max\limits_{0\leq i\leq
q_n}|b_i-a_i|\sum\limits_{i=0}^{q_{n}-1}\int\limits_{a_i}^{b_i}|f''(t)|dt\leq
C\lambda^{n}_1.
\end{equation}
To estimate the sum $\sum\limits_{i=0}^{q_{n}-1}A_{i}$ we rewrite it in the following form:
$$
\sum\limits_{i=0}^{q_{n}-1}A_i=-\sum\limits_{i=0}^{q_{n}-1}
\frac{1}{1+V_i}\left\{\int\limits_{a_i}^{x_i}\frac{f''(t)}{f'(a_i)}\frac{t-a_i}{x_i-a_i}dt
+\int\limits_{x_i}^{b_i}\frac{f''(t)}{f'(a_i)}\frac{b_i-t}{b_i-x_i}dt\right\}=
$$
$$
=-\sum\limits_{i=0}^{q_{n}-1}\int\limits_{a_i}^{b_i}
\frac{f^{\prime
\prime}(t)}{2f^{\prime}(t)}dy-\sum\limits_{i=0}^{q_{n}-1}\left[\frac{1}{f^{\prime}(a_i)(x_i-a_i)}
\int\limits_{a_i}^{x_i}f^{\prime\prime}(t)(t-a_i)dt-
\int\limits_{a_i}^{x_i}\frac{f^{\prime\prime}(t)}{2f^{\prime}(t)}dt\right]-
$$
\begin{equation}\label{sumai}
-\sum\limits_{i=0}^{q_{n}-1}\left[\frac{1}{f^{\prime}(a_i)(b_i-x_i)}
\int\limits_{x_i}^{b_i}f^{\prime \prime}(t)(b_i-t)dt-
\int\limits_{x_i}^{b_i}\frac{f^{\prime\prime}(t)}{2f^{\prime}(t)}dt\right]+
\end{equation}
$$
+\sum\limits_{i=0}^{q_{n}-1}\frac{V_i}{1+V_i}\left[\frac{1}{f^{\prime}
(a_i)(x_i-a_i)}\int\limits_{a_i}^{x_i}f^{\prime\prime}(t)(t-a_i)dt+
\frac{1}{f^{\prime}(a_i)(b_i-x_i)}\int\limits_{x_i}^{b_i}
f^{\prime\prime}(t)(b_i-t)dt\right].
$$
The first sum after the second equality sign gives $(-\log m_n).$ Since $|V_i|=\mathcal{O}(\lambda^{n}_1)$, the absolute value of the last sum in (\ref{sumai}) is bounded above by $C\|f^{\prime\prime}\|\lambda^{n}_1$.
Denote by $S_n$ respectively $\overline{S}_n$, the second and third sums after the last equality sign in (\ref{sumai}). Then we obtain
$$
\sum\limits_{i=0}^{q_{n}-1}A_i=-\log m_n-S_n-\overline{S}_n+O(\lambda^{n}_1).
$$
We rewrite the sum $S_{n}$ in the following form:
$$
S_n=
\sum\limits_{i=0}^{q_{n}-1}\left[\int\limits_{a_i}^{x_i}\frac{f^{\prime
\prime}(t)}{f^{\prime}(t)}\frac{t-a_i}{x_i-a_i}dt-\int\limits_{a_i}^{x_i}
\frac{f^{\prime \prime}(t)}{2f^{\prime}(t)}dt\right]+
$$
$$
+\sum\limits_{i=0}^{q_{n}-1}
\int\limits_{a_i}^{x_i}\frac{f^{\prime\prime}(t)(t-a_i)dt}{f^{\prime}(t)
f^{\prime}(a_i)(x_i-a_i)}
\int\limits_{a_i}^{t}f^{\prime\prime}(s)ds\equiv
S^{(1)}_{n}+S^{(2)}_{n}.
$$
Using Holder's inequlity for the integral we obtain:
$$
|S^{(2)}_{n}|\leq C\cdot
\sum\limits_{i=0}^{q_{n}-1}\left(\int\limits_{a_i}^{b_i}|f''(t)|dt\right)^{2}\leq
C \|f^{\prime\prime}\|_{p}\cdot\max\limits_{0\leq i \leq q_n}|b_i-a_i|^{1/q}
\sum\limits_{i=0}^{q_{n}-1}\int\limits_{a_i}^{b_i}|f''(t)|dt\leq
C \lambda^{n}_2.
$$
where $\lambda_2=\lambda^{\frac{1}{k\cdot q}}$. This together with Proposition \ref{modSn1} imply that
$|S_{n}|\leq \delta_{n}$. Analogously one can show, that $|\overline{S}_{n}|\leq \delta_{n}$. So we get the first estimate in (\ref{modtaun}).

As seen from their definitions in (\ref{Ai}) and (\ref{psitau}), the functions $A_{i}$ and $\psi_{i}$ depend on the variable $x_{i}$, which is linear in the variable $z_{i}$. Therefore $A_{i}$, $\psi_{i}$ themselves depend on $z_{i}$. Calculating the derivatives of $\psi_{i}$ and $A_{i}$ we get
\begin{equation}\label{dphi}
\frac{d\psi_i}{dz_i}=\frac{A^2_i-A'_i}{(1+A_iz_i)(1+A_i(z_i-1))}.
\end{equation}
\begin{equation}\label{dai1}
A'_i=\frac{dA_i}{dz_i}=(b_i-a_i)\frac{dA_i}{dx_i},
\end{equation}
where
\begin{equation}\label{dai}
\frac{dA_i}{dx_i}=\frac{\frac{1}{f'(a_i)(x_i-a_i)^2}
\int\limits_{a_i}^{x_i}f''(t)(t-a_i)dt- \frac{1}{f'(a_i)(b_i-x_i)^2}
\int\limits_{x_i}^{b_i}f''(t)(b_i-t)dt}{1+\frac{1}{f'(a_i)(b_i-a_i)}
\int\limits_{a_i}^{b_i}f''(t)(b_i-t)dt}.
\end{equation}
Consider now
$$
\left|(z_0-z^{2}_0)\frac
{d\tau_{n}(z_0)}{dz_0}\right|=\left|(z_0-z^{2}_0)\sum\limits_{i=0}^{q_{n}-1}
\frac{d\psi_i}{dz_i}\cdot\frac{dz_i}{dz_0}\right|
$$
Since $|A_i|=\mathcal{O}(\lambda^n_1)$, the denominator of the right hand side in (\ref{dphi}) is bounded. Relation (\ref{sumAi2}) implies, that the sum corresponding $A_i^2$ is not greater than $C\lambda^n_1$. As in rewriting $S_n$, we change the  $f''$ (in the integrals in the numerator of $A'_i$ in (\ref{dai})) to $\frac{f''}{f'}$ in the sum $E_n$ (see Remark \ref{modEn}). Then relations (\ref{dai1})-(\ref{dai}), and Lemma \ref{propzi} imply that
$$
\left|(z_0-z^{2}_0)\frac {d\tau_{n}(z_0)}{dz_0}\right|\leq C \lambda^{n}_{1}+ C\sum\limits_{i=0}^{q_{n}-1}(z_i-z^{2}_i)\left|\frac
{dA_{i}}{dz_i}\right|\leq C\lambda^{n}_{2}+E_{n}.
$$
This together with Remark \ref{modEn} imply the second relation in (\ref{modtaun}).

It is clear that
$$
\int\limits_{0}^{1}|\tau'_n(z_0)|dz_0=\int\limits_{0}^{1}\left|\sum\limits_{i=0}^{q_n-1}\frac{d\psi_i}{dz_i}
\frac{dz_i}{dz_0}\right|dz_0\leq
C\lambda^{n}_{1}+\int\limits_{0}^{1}\left|\sum\limits_{i=0}^{q_n-1}\frac{dA_i}{dz_i}\right|dz_{0}.
$$
As when rewriting $S_n$ we change the $f''$  under the integrals in the numerator of $A'_i$ to $\frac{f''}{f'}$ in the last sum. Then using second relation in (\ref{propzi1}), together with (\ref{dai1})-(\ref{dai}) and substituting $z_i=\dfrac{x_i-a_i}{b_i-a_i}$, we get \, $\int\limits_{0}^{1}|\tau'_n(z_0)|dz_0=\mathcal{O}(Q_{n}+\lambda^{n}_1).$ The latter equality and Proposition \ref{modQn} imply the first inequality in (\ref{modtaun1}).

Differentiating (\ref{dphi}), (\ref{dai1}), (\ref{dai}) we obtain:
\begin{equation}\label{d2phi}
\frac{d^2\psi_i}{dz^{2}_i}=\frac{2A_iA^{\prime}_i-A^{\prime\prime}_i}{(1+A_iz_i)(1+A_i(z_i-1))}-
\frac{2(A'_iz_i+A_i)}{1+A_iz_i}\cdot\frac{d\psi_i}{dz_i}-\left(\frac{d\psi_i}{dz_i}\right)^{2}.
\end{equation}
\begin{equation}\label{ai2}
A''_i=\frac{d^2A_i}{dz^2_i}=(b_i-a_i)^2\frac{d^2A_i}{dx^2_i}.
\end{equation}
where
\begin{equation}\label{d2ai}
\frac{d^2A_i}{dx^2_i}=\frac{\int\limits_{a_i}^{x_i}\frac{2(f''(x_i)-f''(t))(t-a_i)}{(x_i-a_i)^2}dt+
\int\limits_{x_i}^{b_i}\frac{2(f''(x_i)-f''(t))(b_i-t)dt}{(b_i-x_i)^2}}{f'(\alpha_i)+
\frac{1}{b_i-a_i}\int\limits_{a_i}^{b_i}f''(t)(t-a_i)dt}.
\end{equation}
We have
$$
\tau''_n(z_0)=\sum\limits_{i=0}^{q_n-1}\left(\frac{d^2\psi_i}{dz^2_i}\cdot
\left(\frac{dz_i}{dz_0}\right)^2+\frac{d\psi_i}{dz_i}\cdot
\frac{dz^2_i}{dz_0^2}\right)
$$
The first relation in (\ref{modtaun1}), relation (\ref{d2phi}) and Lemma \ref{propzi} imply that
$$
\int\limits_{0}^{1}|(z_0-z^2_0)\tau''_n(z_1)|dz_0\leq C \int\limits_{0}^{1}\left|(z_i-z^2_i)
\sum\limits_{i=1}^{q_n-1}(b_i-a_i)^2\frac{d^2A_i}{dx_i^2}\right|dz_{0}+\delta_n\leq
$$
$$
\leq
C\int\limits_{0}^{1}\left|(z_i-z^2_i)\sum\limits_{i=0}^{q_n-1}\left(\frac{b_i-a_i}{x_i-a_i}
\right)^2\int\limits_{a_i}^{x_i}[f''(x_i)-f''(t)]\frac{t-a_i}{x_i-a_i}dt\right|dz_0+
$$
$$
+C\int\limits_{0}^{1}\left|(z_i-z^2_i)\sum\limits_{i=0}^{q_n-1}
\left(\frac{b_i-a_i}{x_i-a_i}\right)^2\int\limits_{x_i}^{b_i}[f''(x_i)-f''(t)]\frac{b_i-t}{b_i-x_i}dt\right|dz_0
+\delta_n
$$
Hence, by Lemma \ref{propzi} and substituting $z_{i}=\dfrac{x_i-a_i}{b_i-a_i}$ in the last integral, we get
$$
\int\limits_{0}^{1}|(z_0-z^2_0)\tau''_n(z_1)|dz_0=\mathcal{O}(U_{n}+\delta_n).
$$
This and Remark \ref{modUn} imply the second relation in (\ref{modtaun1}).
\end{proof}

\section{Proofs of main Theorems}

Before giving the proof of the main results, we approximate relative coordinates $z_{q_{n}}$ by M\"{o}bius functions. Consider an arbitrary fundamental segment $I^{(n)}_{\alpha}=[a; b]$ of the $n$-th basic partition $\mathcal{P}^{(n)}$. Recall, that we have introduced the relative coordinates  $z_i: [f^i(a), f^i(b)]\rightarrow [0, 1]$ by the formula
$$
z_{i}=\frac{f^{i}(x)-f^{i}(a)}{f^{i}(b)-f^{i}(a)},\,\,\,\,\, x\in
[a, b],\,\,\,\,\, 0\leq i\leq q_{n}.
$$
The following lemma shows that $z_{q_{n}}$ is approximated by linear-fractional functions of $z_{0}$, for large $n$.

\begin{lemm}\label{zqnFn} Suppose that $f\in \mathbb{B}^{KO}$. Then the following approximations holds
\begin{equation}\label{apprzqn}
\|z_{q_{n}}-F_{n}\|_{C^{1}([0, 1])}\leq \delta_{n},\,\,\,\, \|z''_{q_{n}}-F''_{n}\|_{L_{1}([0, 1], d\ell)}\leq \delta_{n},
\end{equation}
where $F_n$ is defined in (\ref{Fn}) and $\{\delta_n\}\in l_2$.
\end{lemm}
\begin{proof} In the following we use the following notations:
$$
a_{i}=f^{i}(a),\,\, b_{i}=f^{i}(b),\,\,\, x_{i}=f^{i}(x)\in[a_i,
b_i].
$$
The points $f^i(x)\in I^{(n)}_{\alpha, i}$ are mapped by $f$ to the points $f^{i+1}(x)\in I^{(n)}_{\alpha,
i+1}$, with relative coordinates $z_{i+1}$. Then one has for the relative coordinates $z_i$ and $z_{i+1}$ of the
points $f^i(x)$ respectively $f^{i+1}(x)$ in the interval $[a_i, b_i]$ respectively $[a_{i+1}, b_{i+1}]$:
$$
z_{i}=\frac{x_{i}-a_{i}}{b_{i}-a_{i}},\,\,\,
z_{i+1}=\frac{x_{i+1}-a_{i+1}}{b_{i+1}-a_{i+1}}.
$$
It is clear, that
$$
a_{i+1}=f(a_{i}),\,\,\,\,\, x_{i+1}=f(x_{i})=f(a_{i})+f'(a_{i})(x_{i}-a_{i})+\int\limits_{a_{i}}^{x_{i}}f''(t)(x_{i}-t)dt,
$$
$$
b_{i+1}=f(b_{i})=f(a_{i})+f'(a_{i})(b_{i}-a_{i})+\int\limits_{a_{i}}^{b_{i}}f''(t)(b_{i}-t)dt.
$$
Using this, we rewrite  $z_{i+1}$ as follows
$$
z_{i+1}=\frac{f'(a_{i})(x_{i}-a_{i})+\int\limits_{a_{i}}^{x_{i}}f''(t)(x_i-t)dt}
{f'(a_{i})(b_{i}-a_{i})+\int\limits_{a_{i}}^{b_{i}}f''(t)(b_{i}-t)dt}=
$$
$$
=\frac{x_{i}-a_{i}}{b_{i}-a_{i}}\left(1+\frac{(b_{i}-a_{i})\int\limits_{a_{i}}^{x_{i}}f''(t)(x_{i}-t)dt-
(x_{i}-a_{i})\int\limits_{a_i}^{b_i}f''(t)(b_i-t)dt}{f'(a_i)(b_{i}-a_{i})(x_{i}-a_{i})+
(x_{i}-a_{i})\int\limits_{a_i}^{b_i}f''(t)(b_i-t)dt}\right)=
$$
$$
=z_i(1+A_i(z_i-1)),
$$
where $A_i$ was defined in (\ref{Ai}). It follows that
$$
\frac{1-z_{i+1}}{z_{i+1}}=\frac{1-z_{i}}{z_{i}}\cdot
\frac{1+A_iz_i}{1+A_i(z_i-1)}=\frac{1-z_{i}}{z_{i}}\exp\{N_{i}\}\exp\{-\psi_i\}.
$$
Iterating this equation we obtain
$$
\frac{1-z_{q_n}}{z_{q_n}}=\frac{1-z_{0}}{z_{0}}\exp\left\{\sum\limits_{i=0}^{q_n-1}N_{i}\right\}
\exp\left\{-\sum\limits_{i=0}^{q_n-1}\psi_i\right\}=
\frac{1-z_{0}}{z_{0}}\cdot\frac{1}{m_n\exp(\tau^{(q_n)}(z_0))}.
$$
Solving for $z_{q_n}$ we get
$$
z_{q_{n}}=\frac{z_0m_n e^{\tau_n(z_0)}}{1+z_0(m_n e^{\tau_n(z_0)}-1)}.
$$
A not too hard calculation show, that
$$
z'_{q_{n}}(z_0)=\frac{m_n\exp\{\tau_n(z_0)\}(1+z_0(1-z_0)\tau'_n(z_0))}
{(1+z_0(m_n\exp\{\tau_n(z_0)\}-1))^2},\,\,\, F'_n(z_0)=\frac{m_n}{(1+z_0(m_n-1))^2}.
$$
Then, using the estimates for $\tau_n(z_0)$ in Proposition \ref{estimtaun}, we get the first relation in (\ref{apprzqn}). Similarly,
$$
z''_{q_{n}}(z_0)=\frac{m_n\exp\{\tau_n(z_0)\}(z_0-z^2_0)\tau''_n(z_0)}
{(1+z_0(m_n\exp\{\tau_n(z_0)\}-1))^2}+
$$
$$
+\frac{2m_n\exp\{\tau_n(z_0)\}\left(1-z^2_0-(2z_0-z^2_0)m_n\exp\{\tau_n(z_0)\}\right)\tau'_n(z_0)}
{\left(1+z_0(m_n\exp\{\tau_n(z_0)\}-1)\right)^3}+
$$
$$
+\frac{(1-2z_0m_n\exp\{\tau_n(z_0)\})(z_0-z^2_0)(\tau'_n(z_0))^2}
{(1+z_0(m_n\exp\{\tau_n(z_0)\}-1))^3}-\frac{2m_n\exp\{\tau_n(z_0)\}(m_n\exp\{\tau_n(z_0)\}-1)}
{(1+z_0(m_n\exp\{\tau_n(z_0)\}-1))^3},
$$
$$
F''_n(z_0)=\frac{-2m_n(m_n-1)}{(1+z_0(m_n-1))^3}.
$$
It is clear that the expression $1+z_0(m_n\exp\{\tau_n(z_0)\}-1)$ is bounded and
$$
\int\limits_{0}^{1}|(z_0-z^2_0)(\tau'_n(z_0))^2|dz_0\leq
C\int\limits_{0}^{1}|\tau'_n(z_0)|dz_0.
$$
Then, using  Proposition \ref{estimtaun} and the expression for $z''_{q_{n}}$, we get the result.
\end{proof}

\textbf{Proof of Theorem \ref{main1}.} By definitions of Zoom and the relative coordinates we have
$$
z_{q_n}(z_0)=Z_{I^{(n)}_{\alpha}}(R^{n}(f))(z_0),
$$
where $x=a+z_0(b-a),\,\,\, z_0\in[0, 1]$. Then due to Lemma \ref{zqnFn},  we get Theorem \ref{main1}.

\textbf{Proof of Theorem \ref{main2}.} In \cite{CS2013} an ergodic theorem for the random process, corresponding to a symbolic representation for the elements of partition $\xi_n$, has been proven. Note, that this theorem is also true in our (KO smoothness) case. It follows, that for any $\alpha,\, \beta\in \mathcal{A}$
\begin{equation}\label{ergodic}
\left|\frac{\sum_{f^{i}(I^{n}_{\alpha})\subset
f^{j}(I^{r}_{\beta})}|f^{i}(I^{n}_{\alpha})|}{|f^{j}(I^{r}_{\beta})|}-
\frac{\sum_{i=1}^{q^n_{\alpha}}|f^{i}(I^{n}_{\alpha})|}{|I|}\right|\leq C \lambda^{\sqrt{n}}.
\end{equation}
For simplicity of notion we use $f_n$ to denote $R^{n}(f)$. It is clear, that
\begin{equation}\label{logmn}
-\log \sqrt{m_{n}}=\int\limits_{I^{n}_{\alpha}}\frac{D^{2}f_n(t)}{Df_n(t)}dt.
\end{equation}
Set $r:=[\frac{n}{2}]$. We rewrite the last integral as follows:
$$
\int\limits_{I^{n}_{\alpha}}\frac{D^{2}f_n(t)}{Df_n(t)}dt=\sum\limits_{\beta\in
\mathcal{A}}\sum\limits_{j=1}^{q^{n}_{\beta}}\sum\limits_{f^{i}(I^{n}_{\alpha})\subset
f^{j}(I^{r}_{\beta})}\int\limits_{f^{i}(I^{n}_{\alpha})}\frac{D^{2}f(t)}{Df(t)}dt
$$
Put
$$
\Lambda_n=
\left|\int\limits_{I^{n}_{\alpha}}\frac{D^{2}f_n(t)}{Df_n(t)}dt-
\frac{\sum_{i=1}^{q^n_{\alpha}}|f^{i}(I^{n}_{\alpha})|}{|I|}\int\limits_{[0,
1]}\frac{D^{2}f(t)}{Df(t)}dt\right|.
$$
Then we have
$$
\Lambda_n= \left|\sum\limits_{\beta\in
\mathcal{A}}\sum\limits_{j=1}^{q^{n}_{\beta}}\sum\limits_{f^{i}(I^{n}_{\alpha})\subset
f^{j}(I^{r}_{\beta})}\int\limits_{f^{i}(I^{n}_{\alpha})}\frac{D^{2}f(t)}{Df(t)}dt-
\frac{\sum_{i=1}^{q^n_{\alpha}}|f^{i}(I^{n}_{\alpha})|}{|I|}\sum\limits_{\beta\in
\mathcal{A}}\sum\limits_{j=1}^{q^{n}_{\beta}}\int\limits_{f^{j}(I^{r}_{\beta})}\frac{D^{2}f(t)}{Df(t)}dt\right|\leq
$$
$$
\leq \sum\limits_{\beta\in
\mathcal{A}}\sum\limits_{j=1}^{q^{n}_{\beta}}\left|\sum\limits_{f^{i}(I^{n}_{\alpha})\subset
f^{j}(I^{r}_{\beta})}\int\limits_{f^{i}(I^{n}_{\alpha})}\frac{D^{2}f(y)}{Df(y)}dy-
\frac{\sum_{f^{i}(I^{n}_{\alpha})\subset
f^{j}(I^{r}_{\beta})}|f^{i}(I^{n}_{\alpha})|}{|f^{j}(I^{r}_{\beta})|}
\int\limits_{f^{j}(I^{r}_{\beta})}\frac{D^{2}f(t)}{Df(t)}dt\right|+
$$
$$
+\left|\sum\limits_{\beta\in
\mathcal{A}}\sum\limits_{j=1}^{q^{n}_{\beta}}\left(\frac{\sum_{f^{i}(I^{n}_{\alpha})\subset
f^{j}(I^{r}_{\beta})}|f^{i}(I^{n}_{\alpha})|}{|f^{j}(I^{r}_{\beta})|}-
\frac{\sum_{i=1}^{q^n_{\alpha}}|f^{i}(I^{n}_{\alpha})|}{|I|}\right)
\int\limits_{f^{j}(I^{r}_{\beta})}\frac{D^{2}f(t)}{Df(t)}dt\right|=\Lambda^{(1)}_n+\Lambda^{(2)}_n.
$$
Due to the relation (\ref{ergodic}) we obtain: $\Lambda^{(2)}_n=\mathcal{O}(\lambda^{\sqrt{n}})$. We estimate next the sum
$\Lambda^{(1)}_n$. Denote the endpoints of intervals $f^{i}(I^{n}_{\alpha})$, $f^{j}(I^{r}_{\beta})$ and the ratio of its lengths by
$$
f^{i}(I^{n}_{\alpha})=[a_i,\, b_i],\,\,\,\,
f^{j}(I^{r}_{\beta})=[c_j,\, d_j],\,\,\,\,\rho_{i,
j}=\frac{b_i-a_i}{d_j-c_j},\,\,\,\, 0\leq \rho_{i, j}\leq 1.
$$
We change the variable $y\in[a_i,\, b_i]$ over the first integral in $\Lambda^{(1)}_n$ to $t\in[c_j,\, d_j]$ by the formula: $y=a_i+\rho_{i, j}(t-c_j)$. Then we have
$$
\Lambda^{(1)}_n=\sum\limits_{\beta\in
\mathcal{A}}\sum\limits_{j=1}^{q^{n}_{\beta}}\frac{\sum_{f^{i}(I^{n}_{\alpha})\subset
f^{j}(I^{r}_{\beta})}|f^{i}(I^{n}_{\alpha})|}{|f^{j}(I^{r}_{\beta})|}\left|\int\limits_{f^{j}(I^{r}_{\beta})}
\left(\frac{D^{2}f(a_i+\rho_{i,
j}(t-c_j))}{Df(a_i+\rho_{i, j}(t-c_j))}-
\frac{D^{2}f(t)}{Df(t)}\right)dt\right|
$$
We use the first assertion of Theorem \ref{martin3} to get
$$
\frac{D^{2}f}{Df}-\Phi_0=\sum\limits_{m=1}^{\infty}h_m\;\;
(\mbox{\rm in}\;L_{1}- \mbox{\rm norm}).
$$
By definition, the function $h_{m}(t)$ takes constant values on the atoms of the dynamical partition $\xi_{m}$. On the other hand, $[a_i,\, b_i]\subset [c_i,\ d_i]$. This together with \, $[a_i,\ b_i]\in \xi_{n},\,\,\, [c_j,\ d_j]\in \xi_{r},\,\,\, r=[\frac{n}{2}]<n$ \ imply, that \ $h_{m}(a_i+\rho_{i, j}(t-c_j))=h_{m}(t),\,\, t\in[c_j,\, d_j]$. Next subtracting and adding the sum $\Phi_0+\sum\limits_{m=1}^{N}h_m(t)$ in the integrand in $\Lambda^{(1)}_n$, we obtain:
$$
\Lambda^{(1)}_n\leq 2\sum\limits_{\beta\in
\mathcal{A}}\sum\limits_{j=1}^{q^{n}_{\beta}}\int\limits_{f^{j}(I^{r}_{\beta})}\left|
\frac{D^{2}f(t)}{Df(t)}-\Phi_0-\sum\limits_{m=1}^{N}h_m(t)\right|dt\leq
C\left\|\frac{D^{2}f}{Df}-\Phi_0-\sum\limits_{m=1}^{N}h_m\right\|_{1}
$$
Since
$$
\lim\limits_{N\rightarrow\infty}\left\|\frac{D^{2}f}{Df}-\Phi_0-\sum\limits_{m=1}^{N}h_{m}\right\|_{1}=0,
$$
we can choose a sufficiently large number $N$ such that
$$\left\|\frac{D^{2}f}{Df}-\Phi_0-\sum\limits_{m=1}^{N}h_{m}\right\|_{1}\leq
\lambda^{n}_2.
$$
Consequently, due to relation (\ref{logmn}), we obtain: $|\log m_{n}|=\mathcal{O}(\lambda^{\sqrt{n}})$.

Next define the map $\mathcal{M}_{a}: [0, 1]\mapsto [0, 1]$ as
$$
\mathcal{M}_{a}(x)=\frac{xe^{-\frac{a}{2}}}{1+x(e^{-\frac{a}{2}}-1)},
$$
One can show that the inequality $\|\mathcal{M}_{a}-\mathcal{M}_{b}\|_{C^{2}}\leq C|a-b|$ is fulfilled for every $a, b\in \mathbb{R}$ with $|a|, |b|\leq C$. Using this inequality for $F_{n}$ defined in (\ref{Fn}), we obtain
$$
\|F_n-Id\|_{C^{2}}\leq C|\log m_{n}-0|\leq C\lambda^{\sqrt{n}}.
$$
The last inequality and Theorem \ref{main1} imply the assertions of Theorem \ref{main2}.

\textbf{Afterthought}. At the end of this work, we would like to comment on the further development of our result. It is clear, that the set of k-bounded combinatorics has measure zero. We believe that the same results hold for Roth-type combinatorics which have full measure. Roth-type combinatorics was introduced in \cite{MMY2005}. Katznelson and Ornstein proved, that diffeomorphisms with KO smoothness conditions are absolute continuously conjugated with rigid rotation for irrational rotation numbers of bounded type \cite{KO1989.2}. As mentioned in the introduction, regularity of the conjugation can be obtained by using the convergence of renormalizations of given maps (see e.g. \cite{CS2014}, \cite{KhK2013}, \cite{KhK2014}, \cite{KhK2016}, \cite{KT2013}). Recently we showed in \cite{BC2018} convergence of renormalizations of two maps $f,\, g\in \mathbb{B}^{KO}_{\star}$. Hence it is reasonable to expect absolute continuity of the conjugation between the maps $f$ and $g$.

\textbf{Acknowledgements.} We are grateful to professor Dieter Mayer for useful discussions and comments. The third author (A.D.) was partially supported as a senior associate of ICTP, Italy. The first author (A.B.) is grateful to the Federal University of Bahia for providing with the grant Projeto Capes - PNPD - Matematica UFBA-UFAL. We would like to thank the referee for his careful readings, useful comments and suggestions which helped us to improve the readability of this paper significantly.


\begin{thebibliography}{24}


\bibitem{Ar1961} V.I. Arnol'd: \emph{Small denominators: I. Mappings from the circle onto itself.} Izv. Akad. Nauk SSSR, Ser. Mat., \,\textbf{25},\,21-86\,(1961).

\bibitem{BDM2014} A. Begmatov, A. Dzhalilov and D. Mayer: \emph{Renormalizations of circle homeomorphisms with a single break point}. Disc. \& Cont. Dyn. Syst. - A, Vol. 34, N. 11, 4487- 4513, (2014).

\bibitem{BC2018} A. Begmatov, K. Cunha: \emph{On the convergence of renormalizations of piecewise smooth homeomorphisms on the circle}. https://arxiv.org/abs/1807.09159.

\bibitem{CS2013} K. Cunha, D. Smania: \emph{Renormalization for piecewise smooth homeomorphisms on the circle}, Ann. de l'Institut Henri Poincare (C) Non Lin. Anal., 30(3), 441-462, (2013).

\bibitem{CS2014} K. Cunha, D. Smania: \emph{Rigidity for piecewise smooth homeomorphisms on the circle}, Advan. in Math., 250, 193-226, (2014).

\bibitem{He1979} M. Herman: \emph{Sur la conjugaison diff\'{e}rentiable des diff\'{e}omorphismes du cercle \`{a} des
rotations}. Publ. Math. de L'Inst. des Haut. Scien., \,\textbf{49},\, 5-233,\, (1979).

\bibitem{He1981} M. Herman: \emph{Sur les diffeomorphismes du cercle de nombre de rotation de type constant}. Conf. on Harm. Anal. in Honor of A. Zygmund, Vol. II, 708-725, (1981).


\bibitem{KO1989.1} Y. Katznelson and D. Ornstein: \emph{The differentiability of the conjugation of certain diffeomorphisms of the circle}. Erg. Theo. \& Dyn. Syst., \,\textbf{9},\, 643-680. \,(1989).

\bibitem{KO1989.2} Y. Katznelson and D. Ornstein: \emph{The absolute continuity of the conjugation of certain diffeomorphisms of the circle}. Erg. Theo. \& Dyn. Syst.,\,\textbf{9},\, 681-690,\,(1989).

\bibitem{KhKhm2003} K. Khanin and D. Khmelev: \emph{Renormalizations and Rigidity Theory for Circle Homeomorphisms with Singularities of the Break Type.} Comm. Math. Phys.,\,\,\textbf{235},\,\, 69-124\, (2003).

\bibitem{KhK2013} K. Khanin and S. Koci\'{c}: \emph{Abscence of robust rigidity for circle diffeomorphisms with breaks}. Ann. de l'Inst. Henri Poincare (C) Non Lin. An., 30(3), 385-399, (2013).

\bibitem{KhK2014} K. Khanin and S. Koci\'{c}: \emph{Renormalization conjecture and rigidity theory for circle diffeomorphisms with breaks}. Geom. Funct. Anal.,\,\textbf{24}(6),\,2002-2028\,(2014).

\bibitem{KhK2016} S. Koci\'{c}: \emph{Generic rigidity for circle diffeomorphisms with breaks}, Comm. Math. Phys. 344(2), 427-445, (2016).

\bibitem{KS1987} K.M. Khanin and Ya.G. Sinai: \emph{A New Proof of M. Herman's Theorem}. Comm. Math. Phys., \textbf{112},\, 89-101, (1987).

\bibitem{KS1989} K.M. Khanin and Ya.G. Sinai: \emph{Smoothness of conjugacies of diffeomorphisms of the circle with rotations}. Russ. Math. Surv.,\, \textbf{44},\, 69-99,\, (1989), translation of Usp. Mat. Nauk,\, \textbf{44},\, 57-82,\, (1989).

\bibitem{KT2013} K. Khanin and A. Teplinsky: \emph{Renormalization horseshoe and rigidity for circle diffeomorphisms with breaks}, Comm. Math. Phys. 320, 347-377, (2013).

\bibitem{KT2009} K. Khanin and A. Teplinsky: \emph{Herman's theory revisited}, Invent. math., 178 (2), 333-344, (2009).

\bibitem{KV1991} K.M. Khanin and E.B. Vul: \emph{Circle homeomorphisms with weak discontinuities}. Advan. in Soviet Mathematics,\, \textbf{3},\, 57-98,\, (1991).

\bibitem{KhYam} K. Khanin and M. Yampolsky:  \emph{Hyperbolicity of renormalization of circle maps with a break-type singularity}. Mosc. Math. J. 15, no. 1, 107-121, 182, (2015).

\bibitem{Lan1988} O. Lanford: \emph{Renormalization group methods for critical circle mapping}. Nonlin. evol. and chaotic phenomena. NATO Adv. Sci. Inst. Ser. B: Phys., 176, 25-36 (1988).

\bibitem{MMY2005}S. Marmi, P. Moussa, J.-C. Yoccoz: \emph{The cohomological equation for Roth-type interval exchange maps}, J. Amer. Math. Soc. 18, 823-872, (2005).

\bibitem{MMY2012} S. Marmi, P. Moussa, J.-C. Yoccoz: \emph{Linearization of generalized interval exchange maps}. Ann. of Math., 176, 1583-1646, (2012).

\bibitem{Mo1966} J. Moser: \emph{A rapid convergent iteration method and non-linear differential equations}. II. Ann. Scuola Norm. Sup. Pisa, 20(3), 499-535, (1966).

\bibitem{St1988} J. Stark: \emph{Smooth coniugacy and renormalisation for diffeomorphisms of the circle}. Nonlinearity 1, 541-575,(1988).

\bibitem{Yo1984} J.-C. Yoccoz: \emph{Conjugaison diff\'{e}rentiable des diff\'{e}omorphismes du cercle dont le nombre de rotation v\'{e}rifie une condition diophantienne}. Ann. Scien. de l'Ecole Normale Superieure, 17, 333-361, (1984).



\end{thebibliography}
\end{document}